\documentclass[10pt]{article}
\usepackage[latin1]{inputenc}
\usepackage{a4wide}
\usepackage{amsmath}
\usepackage{amssymb}
\usepackage{amsfonts}
\usepackage{paralist} 
\usepackage{amsthm}
\usepackage{url}
\usepackage{graphicx}
\usepackage{nicefrac}
\usepackage{color}
\usepackage{enumitem}
\usepackage[pdftex]{hyperref}
\usepackage{bbm}
\usepackage{tikz}
\usepackage{times} 
\usepackage[size=small, font=it, labelfont=bf]{caption}
\usepackage[normalem]{ulem} 
\usepackage[numbers,sort&compress]{natbib}
\usepackage{vmargin}

\newcommand{\Prob}[1]{\ensuremath{\mathbb{P} \left(#1 \right)}}
\newcommand{\E}[1]{\ensuremath{\mathbb{E} \left[#1 \right]}}
\newcommand{\Ec}[1]{\ensuremath{\mathbb{E} [#1]}}
\newcommand{\cO}{\ensuremath{\mathcal F}} 

\newcommand{\SW}{\ensuremath{\mathrm{SW}}}
\newcommand{\NW}{\ensuremath{\mathrm{NW}}}
\newcommand{\SE}{\ensuremath{\mathrm{SE}}}
\newcommand{\NE}{\ensuremath{\mathrm{NE}}}

\newtheorem{thm}{Theorem}[section]
\newtheorem{lem}{Lemma}[section]
\newtheorem{cor}[thm]{Corollary}
\newtheorem{rem}{Remark}[section]
\newtheorem{proposition}[thm]{Proposition}

\newcommand{\Ck}{\ensuremath{\mathcal{C}_k}}
\newcommand{\Cfour}{\ensuremath{\mathcal{C}_4}}
\newcommand{\Ctwo}{\ensuremath{\mathcal{C}_2}}
\newcommand{\Cone}{\ensuremath{\mathcal{C}_1}}
\newcommand{\vol}{\ensuremath{\mathrm{Vol}}}

\hypersetup{
    unicode      = false,     
    pdftoolbar   = true,      
    pdfmenubar   = true,      
    pdffitwindow = true,      
    pdfnewwindow = true,      
    colorlinks   = true,      
    linkcolor    = blue,      
    citecolor    = red,      
    filecolor    = blue,      
    urlcolor     = blue       
}

\begin{document}

\title{\bf A limit field for orthogonal range searches in two-dimensional random point search trees}
\author{
Nicolas Broutin 
\thanks{Sorbonne Universit\'e, Campus Pierre et Marie Curie
Case courrier 158, 4, place Jussieu 
75252 Paris Cedex 05 
France. {\bf Email}: nicolas.broutin@upmc.fr; {\bf Grant}: ANR-14-CE25- 0014 (ANR GRAAL)} 
\and
Henning Sulzbach 
\thanks{University of Birmingham, School of Mathematics, B15 2TT, Birmingham, UK.
{\bf Email}: henning.sulzbach@gmail.com. The author's research was partially supported by a Feodor Lynen Fellowship of the Alexander von Humboldt-Foundation.}
}

\date{\today} 
\maketitle

\begin{abstract}
We consider the cost of general orthogonal range queries in random quadtrees. The cost of a given query is encoded into a (random) function of four variables which characterize the coordinates of two opposite corners of the query rectangle. We prove that, when suitably shifted and rescaled, the random cost function converges uniformly in probability towards a random field that is characterized as the unique solution to a distributional fixed-point equation. We also state similar results for $2$-d trees. Our results imply for instance that the worst case query satisfies the same asymptotic estimates as a typical query, and thereby resolve an open question of Chanzy, Devroye and Zamora-Cura [\emph{Acta Inf.}, 37:355--383, 2001].  

\end{abstract} 



\section{Introduction}

\subsection{Quadtrees and structures for geometric data}

Geometric data are central in a number of practical contexts, such as computer graphics, management of geographical data or statistical analysis. Data structures storing such information should allow for efficient dictionary operations such as updating the data base and retrieving data matching specified patterns. For general references on multidimensional data structures and more details about their various applications, see the series of monographs by \citet{Samet1990a,Samet1990,Samet2006}. 

We are interested in tree-like data structures which permit efficient execution of search queries. In applications, one of the essential basic type of queries type are \emph{(orthogonal) range queries}, which ask to report all data located inside some axis parallel rectangular region. Such queries include the case when some of the projections of the rectangular region on the axes are either reduced to a point or span the entire domain. So, in particular, range queries cover the following two cases: 
\begin{itemize}
    \item When the pattern specifies precisely all the data fields (the query rectangle is a point), we speak of an \emph{exact match}. Such queries can typically be answered in time logarithmic in the size of the database, since only one branch needs to be explored. 
   
    \item When the projections of the query rectangle on the different axes are either points or the entire domain, we speak of a \emph{partial match}. In general, such searches explore multiple branches of the data structure to report the matching data, and the cost usually becomes polynomial. 
\end{itemize}

We are interested in the comparison-based setting, where the data may be compared directly at unit cost. In this context, a few general purpose data structures generalizing binary search trees permit to answer orthogonal range queries, namely the  quadtree \cite{FiBe1974}, the $k$-d tree \cite{Bentley1975} and the relaxed $k$-d tree \cite{DuEsMa1998}. Since range queries are at the same time an essential building block of many other algorithms and still rather elementary, one would expect that their complexity should be fully understood by now. This is not the case: despite their importance, a precise quantification of the complexity of range queries in the data structures listed above is still missing. We will shortly review the literature precisely, but let us for now point out that, before the present document, even the average value in the context of uniformly random data points and a uniformly random query was only known up to a multiplicative factor.

In this paper, we provide refined analyses of the costs of orthogonal range queries in the two-dimensional data structures mentioned earlier. We mostly focus on quadtrees, but our results also apply (modulo some easy model-specific modifications) to the case of $2$-d trees and  relaxed $2$-d trees. We only sketch the results for $2$-d trees,  since the phenomena at hand and the proofs are completely analogous, and the cases of partial match queries have been treated in \cite{BrNeSu13} (see Section~\ref{sec:kd});  the case of relaxed $2$-d trees is also similar, but we leave it as an exercise to keep the present paper as concise as possible. Our results provide, for the first time, a study of the extent of the fluctuations of the complexity, precise asymptotic estimates for all moments and convergence in distribution, jointly for all axis-parallel rectangular queries.

\medskip
\noindent\textbf{Remarks.} We emphasize the fact that we are interested in \emph{general purpose data structures}: for instance, the range trees of Bentley \cite{Bentley1979a} (and their close relatives, see \cite{BeFr1979}) are complex data structures tailored to answer range queries very fast, but this is mostly a theoretical benchmark since the space required is super-linear in the number of data points. On the other hand, the squarish $k$-d trees introduced by \citet*{DeJaZa2001} would fit; unfortunately, while it is very interesting, our results do not apply to these trees. We plan to address the complexity of orthogonal range queries in squarish $k$-d trees in the near future.

\medskip
Before going any further, let us introduce two-dimensional quadtrees and range queries. Given a sequence of points (also called \emph{keys}) $(p_i)_{i\ge 1} \in [0,1]^2$ we define a sequence of quadtrees $(T_n)_{n\ge 1}$, together with a recursive partition of the unit square. We can see the construction algorithm as inserting the points successively and ``storing'' them in a node of a quaternary tree. We proceed as follows: Initially, we think of $T_0$ as an empty tree, which consists of a placeholder to which we assign the unit square. The first point $p_1$ is inserted in this placeholder and becomes the root, thereby giving rise to four placeholders. Geometrically, $p_1$ decomposes the unit square into four rectangular regions $Q_1, \ldots, Q_4$, situated, in this order, south-west, north-west, south-east and north-east of $p_1$, each of which is assigned to a child of the root (currently placeholders). Suppose that we have constructed $T_n$ by successive insertion of $p_1, \ldots, p_n$, and that $T_n$ induces a partition of the square into $1+3n$ rectangles, each one assigned to one of the $1+3n$ placeholders of $T_n$. The next point $p_{n+1}$ is then placed in the placeholder, say $v$, that is assigned to the rectangle containing $p_{n+1}$. This operation turns $v$ into a node and creates four new placeholders just below. Geometrically, $p_{n+1}$ divides this rectangle into four subregions that are assigned to the four newly created placeholders. See Figure~\ref{fig:const} for an illustration of the first few steps.
\begin{figure}[t]   
\centering
\includegraphics[width=14cm]{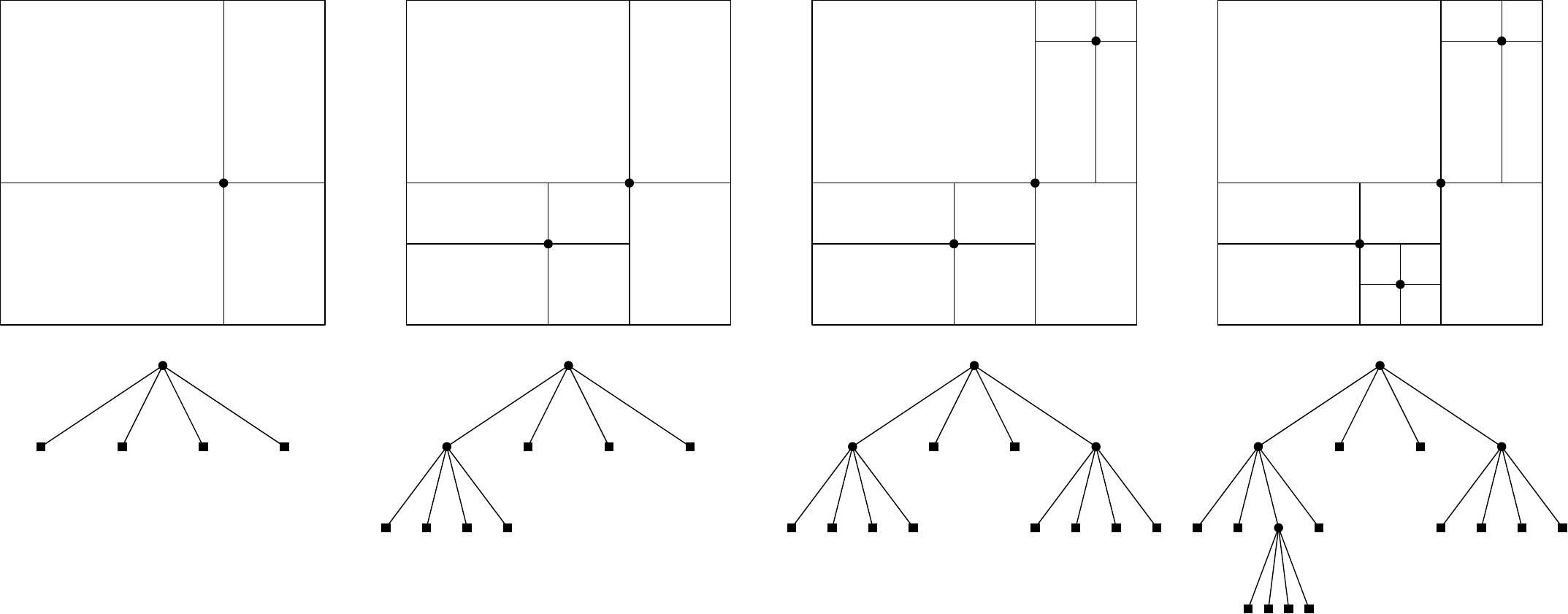} 
\caption{The first four steps in the quadtree construction of the partition, and of the corresponding trees. The placeholders are marked with squares while the actual nodes, which `store' one of the points, are depicted with disks.}
\label{fig:const} 
\end{figure}  

An orthogonal range query asks to retrieve all points among $p_1, \ldots, p_n$ in $T_n$ which are contained inside a given rectangle with (opposite) corners $(a,c)$ and $(b,d)$. This covers partial match queries, aiming at all elements with a given value for the first coordinate, regardless of the second one (taking the degenerate rectangle with corners at $(a,0)$ and $(a,1)$). Further, the set-up also includes a fully specified query aiming at determining whether $T_n$ contains the element $(a,c)$ (taking the degenerate rectangle with corners $(a,c)$ and $(a,c)$). The search algorithm explores recursively the nodes of the tree corresponding to regions that have a non-empty intersection with the query (see later on for a precise definition of the regions and query, and in particular of which of their boundaries are included). As a basic measure of complexity, we consider the number of nodes inspected by the algorithm, where placeholders are not taken into account. (See Figures~\ref{fig:queries} and~\ref{fig:quadtree_query} for illustrations of the different query types and  Remark \ref{rem:ext} in Section \ref{sec:strategy} for the model incorporating placeholders.)

\subsection{The cost of queries in random quadtrees }

In a standard probabilistic model, the quadtree $T_n$ is constructed from the first $n$ points $X_1,  \ldots, X_n$ of a single infinite sequence of i.i.d.\ random variables $(X_i)_{i\ge 1}$ with uniform distribution on $[0,1]^2$. In the notation introduced above, we set $p_i  = X_i$ for all $i \geq 1$. Complexities of range queries are now random variables and one is mostly interested in their average behaviour and fluctuations but also in distributional limit theorems.


\medskip \noindent\textbf{Fully specified queries and distances.} 
Here, the cost corresponds to the number of nodes on the search path, or, equivalently to the number of rectangles that contain the query point in the full refining family of nested partitions. In our probabilistic model, fully specified queries are well understood. It is known that the corresponding mean complexity is asymptotic to $\log n$ \cite{DeLa1990, FlGoPuRo1993}, the variance grows like $\frac 1 2 \log n$ \cite{FlGoPuRo1993}, and that there is a central limit theorem \cite{FlLa1994, Devroye1998}. For all these results, note that this essentially corresponds to the cost of a \emph{random} query, but adding the reference to a point in $[0,1]^2$ would not make a difference (except on the edges, where the leading constant is different). 
The extreme cost is the height of the tree, and Devroye has proved that it is asymptotic to $\alpha \log n$, where $\alpha = 2.15\ldots$ \cite{Devroye1987}. 

\medskip \noindent \textbf{Partial match queries.}
The history case of partial match queries stretches over a much longer period and has only been finely determined very recently \cite{BrNeSu2011a,BrNeSu13}. For $t \in [0,1]$, let $C_n(t)$ denote the number of nodes visited by a query in $T_n$ retrieving all keys with first coordinate $t$. Equivalently, $C_n(t)+1$ is given by the number of rectangles in the partition of $[0,1]^2$ induced by $T_n$ intersecting the vertical line at $t$. See Figure~\ref{fig:queries} for an illustration. Throughout the document, $\xi$ denotes a generic random variable with uniform distribution on $[0,1]$ which is stochastically independent of $(X_i)_{i\ge 1}$. In their seminal work on properties of 
random quadtrees, \citet*{FlGoPuRo1993} considered the case of random queries and showed that
\begin{align} \label{limit_unif} \E{C_n(\xi)} = \kappa n^{\beta} + O(1), \quad n \to \infty, \end{align}
where
\begin{align} 
\label{limit_unif_const} 
\beta = \frac{\sqrt{17} -3}{2}, \qquad \text{and} \qquad \kappa = \frac{\Gamma(2\beta+3)}{2\Gamma(\beta+1)^3}, 
\end{align}
and $\Gamma(\,\cdot\,)$ denotes the Gamma function. In fact,  \citet*{FlGoPuRo1993} provide a full asymptotic expansion for $\Ec{C_n(\xi)}$ which was generalized to the higher dimensional case by \citet{ChHw2003}. In the early 2000s, using the recursive approach underlying the results obtained by \citet{FlGoPuRo1993}, there were some attempts to obtain more detailed asymptotic information such as the variance or the limit distribution of $C_n(\xi)$. However, it turns out that for a random query $\xi$, there is no obvious recursion for higher moments since the query couples the subproblems (the location is the same in subproblems!). To circumvent this issue, one can consider fixed queries. It is only recently that this route has been successfully explored: Curien and Joseph \cite{CuJo2010} proved that, for fixed $t \in [0,1]$, as $n \to \infty$,
\begin{align} 
\label{limit_t}
\Ec{C_n(t)} = K_1 h(t) n^{\beta} + o(n^{\beta}), 
\end{align}
where   
\begin{align} 
\label{limit_t_const} 
h(t) = (t(1-t))^{\beta/2}, 
\qquad \text{and} \qquad 
K_1= \frac{\Gamma(2\beta+2)\Gamma(\beta+2)}{2\Gamma(\beta+1)^3\Gamma(\beta/2+1)^2}. 
\end{align}
In the joint work \cite{BrNeSu13} (see also \cite{BrNeSu2011a}) with Ralph Neininger relying on the functional contraction method developed in \cite{NeSu}, we established a distributional functional limit theorem for the rescaled process $n^{-\beta} C_n$:  upon considering $C_n$ as a right-continuous step function on $[0,1]$, we have the following convergence in distribution
\begin{align} 
\label{lim_fun} 
n^{-\beta} C_n \xrightarrow[n\to\infty]{d} {\cal Z},
\end{align}
where the limit process $\cal Z$ is continuous and is further discussed in Section~\ref{sec:partial_match} below. Here, the convergence is in the space of c{\`a}dl{\`a}g functions on the unit interval. See \cite[Section 3]{Billingsley1999} for background. Based on this result, Curien \cite{cucpc} was able to show that the convergence actually holds in probability (and, for any fixed $t \in [0,1]$, even almost surely). 
\begin{figure}[tb]
    \centering
    \includegraphics[scale=.55]{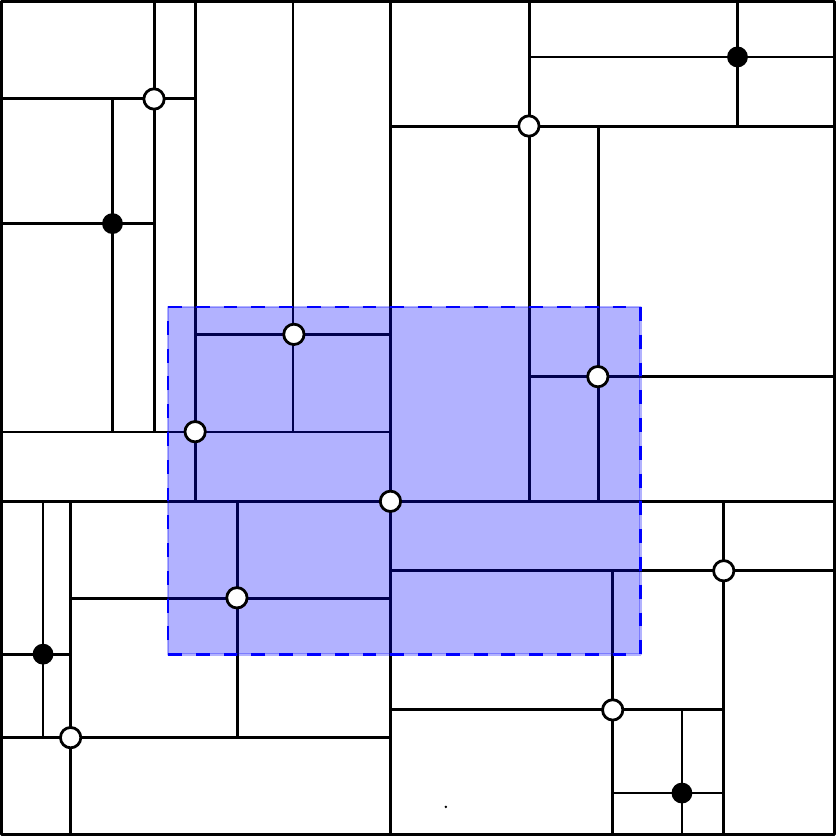}\hspace{1em}
    \includegraphics[scale=.55]{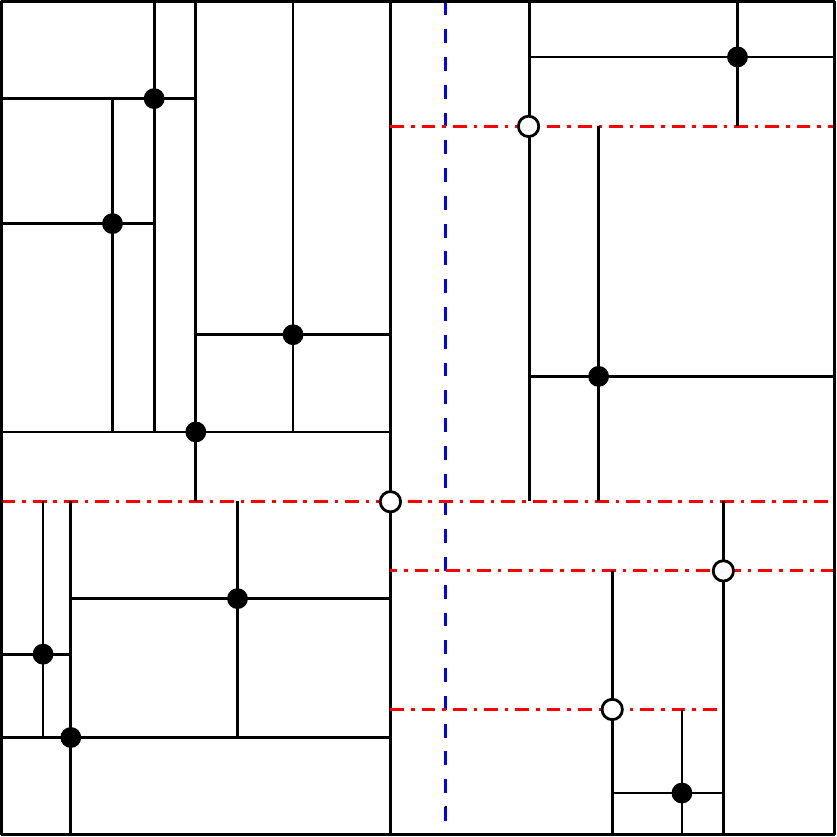}\hspace{1em}
    \includegraphics[scale=.55]{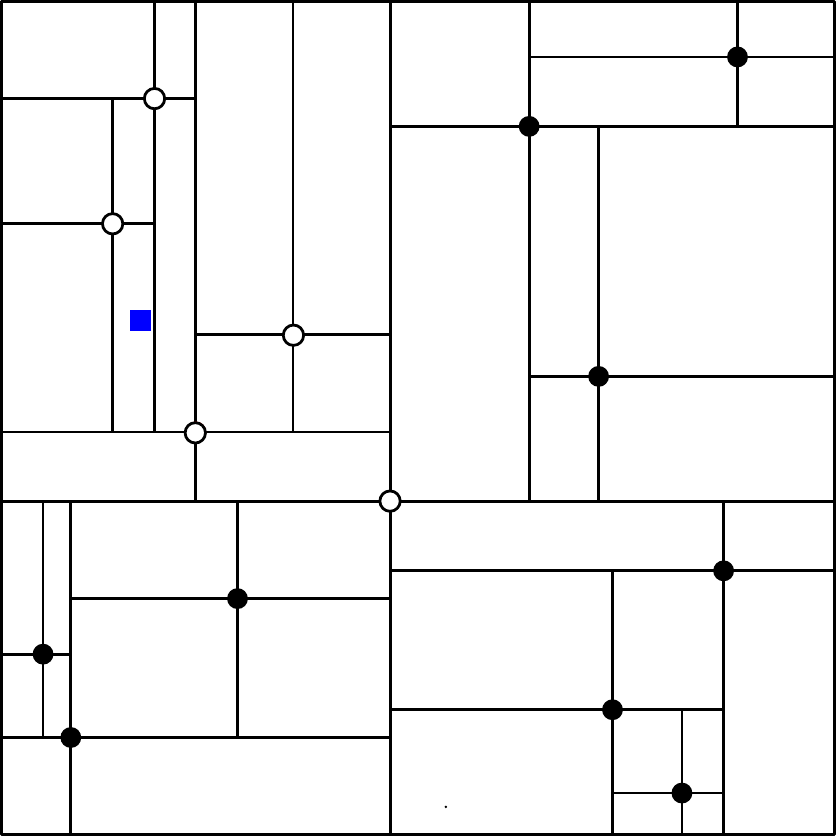}
    \caption{The partition induced by a quadtree and the different kinds of range queries. Each time, the query rectangle (which might be lower dimensional) is depicted in blue (with dashed boundary, except for the exact match on the right where it is the square point), and the nodes visited are shown in white: on the left a non-degenerate range query; in the middle, a partial match query, and on the right, a fully specified query. The corresponding tree is depicted in Figure~\ref{fig:quadtree_query}.}
    \label{fig:queries}
\end{figure}
\begin{figure}[tb]
    \centering
    \includegraphics[width=0.95\textwidth]{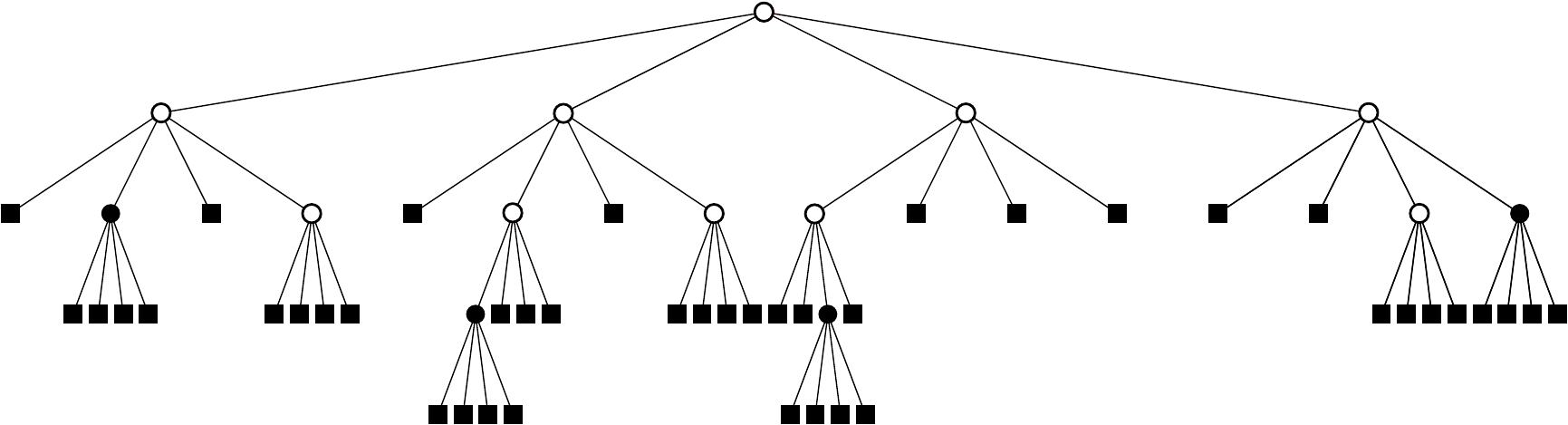}
    \caption{The quadtree corresponding to the partition shown in Figure~\ref{fig:queries}. The nodes in white are the nodes visited by the range query illustrated in the left picture there, and correspond to the nodes in white in that picture. Observe that the placeholders 
    are not counted.}
    \label{fig:quadtree_query}
\end{figure}

\medskip
\noindent \textbf{Orthogonal range queries.} Despite their importance, the cost of range searches has been much less studied. For the closely related structure of $k$-d trees, the first important contribution is due to \citet*{ChDeZa2000} who consider queries anchored at a uniformly random point in the unit square with specific predetermined dimensions $\Delta_1$ and $\Delta_2$ along the first and second coordinates. (These queries are allowed to exit $[0,1]^2$.) For the average cost $m_n$ of such a query, they obtain the following explicit bounds 
\begin{equation}\label{eq:bounds_DeChZa}
\gamma \le \frac{m_n} {\Delta_1 \Delta_2 n + (\Delta_1 + \Delta_2) n^\beta + \log n} \le \gamma',  \quad 0 < \gamma < \gamma' < \infty.
\end{equation}
This exhibits a contribution of the ``volume'' of nodes to report, and a ``perimeter'' effect of the order of magnitude of a partial match. \citet*[Section 8]{ChDeZa2000} also mention that an analogous result should hold for random quadtrees.


\medskip
The aim of the present paper is to prove a limit theorem for the joint complexity of all range searches simultaneously, just as \eqref{lim_fun} for the partial match queries. Such an approach will also give access to the cost of extreme queries, which depend on the data points in an intricate way. We first define the set of all queries for which one should have joint convergence. Let $I = \{(a,b,c,d) \in [0,1]^4 : a \leq b, c \leq d\}$. For $(a,b,c,d)\in I$, let $Q(a,b,c,d) = (a,b] \times (c,d]$ and denote by $O_n(a,b,c,d)$ the number of nodes visited by the algorithm to answer the query with rectangle $Q(a,b,c,d)$ in $T_n$. Formally, letting $\mathcal R_i$ denote the rectangular region associated with the placeholder in $T_{i-1}$ in which the point $X_i$ falls, we have
\begin{align} O_n(a,b,c,d) = \sum_{i=1}^n \mathbf 1_{\mathcal R_i \cap Q(a,b,c,d) \neq \emptyset}. \label{def:On} 
\end{align}
This is well-defined on the set of full Lebesgue measure for which $a,b,c$ and $d$ differ from
all the coordinates of the $X_i$'s regardless of the precise definition of the $\mathcal R_i$ 
(which is given later in \eqref{defRformal} in Section~\ref{sec:partial_match});
we agree to extend it by imposing right continuity in all coordinates. Note that, by this convention, the range query is a genuine generalization of partial match queries since, for $t\in [0,1]$, we have $O_n(t,t,0,1) = C_n(t)$. Our main results are presented in the following section.

\subsection{Main results: joint convergence of all orthogonal range queries}
\label{sec:results}

We are interested in the joint asymptotic cost of all range queries with rectangle $Q(a,b,c,d), (a,b,c,d)\in I$, in $T_n$, that is, in the asymptotic behavior of the family of random variables $O_n(a,b,c,d), (a,b,c,d)\in I$. Let $\Cfour^+$ be space of continuous functions on $I$
equipped with supremum norm $\| \cdot \|$ such that $\|f\| = \sup_{(a,b,c,d) \in I } |f(a,b,c,d)|$.
For $(a,b,c,d)\in I$, let $\vol(a,b,c,d)=(b-a)(d-c)$ denote the area of the query rectangle $Q(a,b,c,d)$.
Recall that the sequence of random fields $O_n = O_n(a,b,c,d), n \geq 1$ relies on the sequence of quadtrees  $(T_n)_{n\ge 1}$ constructed from  the same sequence $(X_i)_{i\ge 1}$. Our main result is the following theorem. 

\begin{thm} \label{thm_main}
There exists a random continuous $\Cfour^+$-valued random variable $\cO$ (a random field) such that, in probability and with convergence of all moments,
\[\Big \| \frac{O_n - n \vol }{n^\beta} - \cO\Big \| \xrightarrow[n\to\infty]{} 0\,.\]
\end{thm}
Observe that the statement in Theorem~\ref{thm_main} also covers query rectangles with zero Lebesgue measure ($\vol(a,b,c,d)=0$). The following straightforward consequence settles the question of \citet*{ChDeZa2000} about the average cost of the worst-case range query:
\begin{cor}We have the following convergence, in probability with convergence of all moments:
\[n^{-\beta} \cdot \sup_{(a,b,c,d)\in I} \Big\{ O_n(a,b,c,d) - n(b-a)(d-c) \Big\} \xrightarrow[n\to\infty]{} \sup_{(a,b,c,d) \in I} \cO(a,b,c,d)\,.\]
\end{cor}

The limit field $\cO$ is uniquely characterized as the solution to a stochastic fixed-point equation.

\begin{proposition} \label{thm2}
Up to a multiplicative constant, the process $\cO$ is the unique $\Cfour^+$-valued random field (in distribution) with $\Ec{\|\cO\|^2} < \infty$ satisfying the stochastic fixed-point equation
\begin{align} \label{fix:O} 
\cO \stackrel{d}{=} \sum_{r=1}^4 D_r(\cO^{(r)}),
\end{align}
where $\cO^{(1)}, \ldots, \cO^{(4)}$ are copies of $\cO$, $D_1, \ldots, D_r$ are random linear operators defined in \eqref{opD} and the random variables $\cO^{(1)}, \ldots, \cO^{(4)}$, and $(D_1, \ldots, D_4)$ are independent.
\end{proposition}

Proposition~\ref{thm2} only characterizes the distribution of $\cO$ up to a multiplicative constant. The following proposition identifies the limit mean, and hence the missing multiplicative constant. 


\begin{proposition}\label{prop:mean_main} Let $(a,b,c,d)\in I$. Then, we have
\[\Ec{\cO(a,b,c,d)} = \frac 1 2 \Big( \mu(a,d) - \mu(a,c) + \mu(b,d) - \mu(b,c) + \mu(c,b) - \mu(c,a) + \mu(d,b) - \mu(d,a)\Big),\]
where $\mu(t,s) = K_1 h(t) g(s)$, the function $h(t)$ and the constant $K_1$ are defined in \eqref{limit_t_const}, and $g$ is a continuous and monotonically increasing bijection on $[0,1]$ satisfying $g(s) = 1 - g(1-s)$ for every $s \in [0,1]$. Furthermore $g$ is ${\cal C}^\infty$ on $(0,1)$ and the unique bounded measurable function on $[0,1]$ satisfying, for every $s\in [0,1]$,
\begin{align} \label{fix:g}  
g(s) = \frac{\beta + 1}{2} \left( \int_s^1 v^\beta g\!\left(\frac s v \right) dv + \int_0^s (1-v)^\beta g\!\left(\frac {s-v} {1-v} \right) dv \right) + \frac 1 2 s^{\beta+1}\,.
\end{align}
\end{proposition}
\begin{figure}
\centering
\includegraphics[scale=.4]{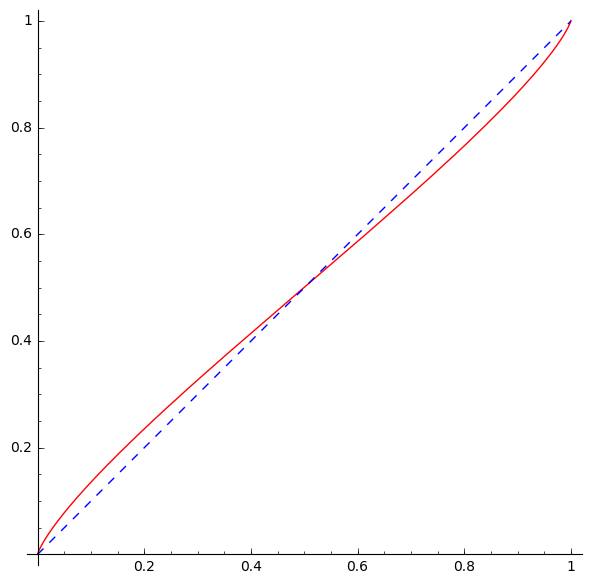}
\caption{\label{fig:iteration_g} The function $g$ obtained by iteration in solid line, and the identity (dashed) for comparison.} 
\end{figure}

\medskip
\noindent \textbf{About the higher-dimensional case.} All results in the present paper crucially rely on the limit theorems for partial match complexities formulated in \cite{BrNeSu13}, whose proofs are based on the functional contraction method developed in \cite{NeSu}. While high-dimensional analogues of the results summarized on fully specified queries, the mean complexity of partial match queries \eqref{limit_unif} and inequalities of type \eqref{eq:bounds_DeChZa} are known \cite{DeLa1990, FlGoPuRo1993, FlLa1994, Devroye1998, FlLaLaSa1995, ChHw2003, ChDeZa2000}, 
generalizations of our results would require to extend the contraction method to functions of multiple variables which is technically more demanding. For instance, even the high-dimensional analogues of the results in \cite{BrNeSu13} are unknown. We intend to carry out an analysis of partial match and range queries in the $d$-dimensional case, for $d\ge 3$, elsewhere. Here, it is important to observe that, while we are dealing with random functions of multiple variables, we never use directly the contraction method for these fields thanks to a number of couplings. For the curious reader, we conjecture that Theorem \ref{thm_main} extends to the $d$-dimensional case where $O_n$ stands for the cost of a query given by a $d$-dimensional cylinder, the limit field is parametrized by $2d$ coordinates, and  $\beta$ is to be replaced by a constant $1/2 < \beta_d < 1$ governing the complexity of partial match queries in $d$-dimensional quadtrees when one data field is fixed and $d-1$ data fields are unspecified. (By \cite[Theorem 5]{FlGoPuRo1993}, $\beta_d$ is the unique solution to the equation
$(x+2)(x+1)^{d-1} = 2^d$ on the unit interval.)


\subsection{Strategy of the proof and combinatorial identities}
\label{sec:strategy}

The proof of Theorem~\ref{thm_main} relies on a representation of $O_n(a,b,c,d)$ as a sum of different terms which aims at accounting precisely for the contributions of the ``volume'' and ``perimeter'' effects that are visible in the upper and lower bounds in \eqref{eq:bounds_DeChZa}. 

Every point $p=(t,s)\in [0,1]^2$ partitions the unit square into four regions  
\begin{align*}
\SW(p) &= [0,t] \times [0,s]\\
\NW(p) &= [0,t] \times (s,1]\\
\SE(p) &= (t,1] \times [0,s]\\
\NE(p) &= (t,1] \times (s,1]\,.
\end{align*}
Let $R_1=\SW(a,c)=[0,a]\times [0,c]$, $R_2=\NW(a,d)=[0,a]\times (d,1]$, $R_3=\SE(b,c)=(b,1]\times [0,c]$, $R_4=\NE(b,d)=(b,1]\times (d,1]$, be the regions south-west, north-west, south-east and north-east of $Q(a,b,c,d)$, respectively. (See Figure~\ref{fig:decomp} for an illustration.) Let also $S_1=[a,b]\times [0,c]$, $S_2=[0,a]\times (c,d]$, $S_3=(b,1]\times (c,d]$ and $S_4=(a,b] \times (d,1]$ denote the regions south, east, north and west of $Q(a,b,c,d)$. Then $[0,1]^2$ is the disjoint union of $Q(a,b,c,d)$, $R_1,R_2,R_3,R_4$ and $S_1,S_2, S_3, S_4$. 

Fix $n\ge 1$. Let $N_n(a,b,c,d)$ denote the number of points among $X_1, X_2, \ldots, X_n$ that lie within the query rectangle $Q(a,b,c,d)$. For $s,t\in [0,1]$, let $Y_n^<(t,s)$ denote the number of nodes visited to answer the partial match query $\{x=t\}$ such that the corresponding points in $[0,1]^2$ lie in $\SW(t,s)$. So, for instance, the number of points lying in $S_2$ that are visited when answering the partial match query $\{x=a\}$ is $Y^<_n(a,d)-Y^{<}_n(a,c)$. Similarly, define $Y_n^{\ge }(t,s)$ as the number of points lying in $\SE(t,s)$ that are visited by a partial match query at $t$. The functions $\bar Y^{<}_n(t,s)$ and $\bar Y^{\ge }_n(t,s)$ are defined in a symmetric way when exchanging the first and second coordinates of every point $X_1,X_2,\dots, X_n$.

Let $D^{(1)}_n(a,b,c,d)$ denote the number of points among $X_1,X_2,\dots, X_n$ that are lying in $R_1$ and are visited by a fully specified search query in $T_n$ retrieving $(a,c)$. Analogously, let $D^{(2)}_n(a,b,c,d)$, $D^{(3)}_n(a,b,c,d)$ and $D^{(4)}_n(a,b,c,d)$ denote the number of points lying respectively in $R_2$, $R_3$ and $R_4$ that are visited to answer a fully-specified query for $(a,d)$, $(b,c)$ and $(b,d)$, respectively.

We agree that, on horizontal and vertical lines containing points in the set $\{X_1, X_2, \ldots\}$, all functions introduced in the previous paragraph are continuous from the right in all coordinates. 

The following representation is reminiscent of the decomposition in sliced queries by \citet{DuMa2002a}.

\begin{lem}\label{lem:decomp}
Fix $n\ge 1$ and a quadtree on $n$ points. For $(a,b,c,d)\in I$, we have
\begin{align}\label{eq:decomp}
O_n(a,b,c,d) = & ~ N_n(a,b,c,d) \notag \\
& + Y^\geq_n(b,d) -Y^\geq_n(b,c) + Y^<_n(a,d) -Y^<_n(a,c) \notag \\
& + \bar Y ^\geq_n(d,b) - \bar Y^\geq_n(d,a) + \bar Y^<_n(c,b) - \bar Y^<_n(c,a) \notag \\
& + D^{(1)}_n(a,b,c,d) + D^{(2)}_n(a,b,c,d) + D^{(3)}_n(a,b,c,d) +D^{(4)}_n(a,b,c,d).
\end{align}
 \end{lem}
 \begin{figure}[htbp]
     \centering
     \includegraphics[width=0.4\textwidth]{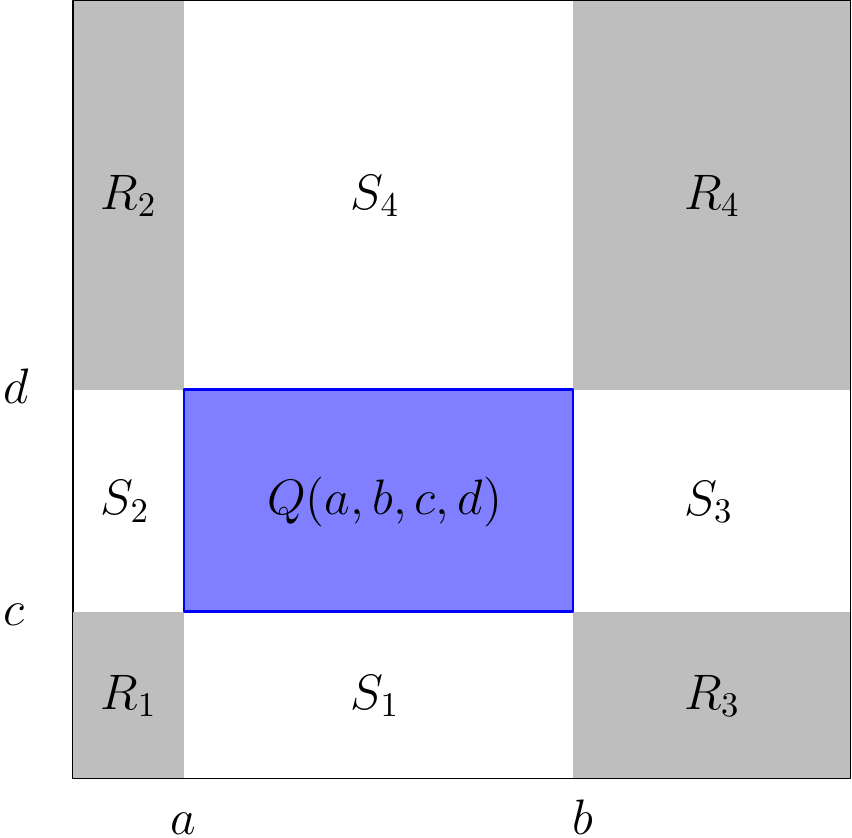}
     \caption{The decomposition in Lemma~\ref{lem:decomp} of the points visited to answer the query with rectangle $Q(a,b,c,d)$, here depicted in blue, and which are counted by $O_n(a,b,c,d)$.}
     \label{fig:decomp}
 \end{figure}
 
\begin{proof}Consider $X_1,X_2,\dots, X_n$ as fixed, and fix $(a,b,c,d)\in I$. The points that are visited to answer some range query may be grouped into different subsets depending on to which part of the unit square among $Q(a,b,c,d)$, $R_1,R_2,R_3, R_4$ or $S_1,S_2,S_3,S_4$ they belong. 

Now, by \eqref{def:On}, a node of $T_n$ storing the point $X_i$ is visited when answering the query with rectangle $Q(a,b,c,d)$ if and only if, at the time of its insertion in $T_n$, the point $X_i$ falls in a rectangle of the partition induced by the quadtree $T_{i-1}$ that intersects $Q(a,b,c,d)$. It follows that the nodes corresponding to the points $X_i$ lying within $Q(a,b,c,d)$ are all visited. Furthermore, the nodes corresponding to the points lying in $R_1$, $R_2$, $R_3$ and $R_4$ are visited when answering the query $Q(a,b,c,d)$ if and only if they are visited when answering the fully specified queries for the points $(a,c)$, $(a,d)$, $(b,c)$ and $(b,d)$, respectively. Finally, we need to consider the points among $X_1,X_2,\dots, X_n$ lying in the rectangles $S_1$, $S_2$, $S_3$ and $S_4$. For these, the rectangle corresponding to a point $X_i$ intersects $Q(a,b,c,d)$ if and only if it intersects the line segment separating $S_i$, $1\le i\le 4$, from $Q(a,b,c,d)$ (open at one of the end points). It follows easily that, aside from the points that may lie on one of the boundaries of the regions, the contributions are precisely the number of points lying in $S_1$, $S_2$, $S_3$ or $S_4$ reported by one of the partial match queries $\{x=a\}$, $\{x=b\}$, $\{y=c\}$, or $\{y=d\}$.  If follows that the claim holds for all $(a,b,c,d)\in I$ provided that no data point lies on any of these lines; the right-continuity of the functions completes the proof.
\end{proof}

With Lemma~\ref{lem:decomp} under our belt, it is now relatively easy to get an intuition for why Theorem~\ref{thm_main} should hold, and what is needed to turn this intuition into a proof. On the one hand, one observes that $N_n(a,b,c,d)$ is distributed as a binomial random variable with parameters $n$ and $\vol(a,b,c,d)$, so that $N_n(a,b,c,d) = n \vol(a,b,c,d) +O(\sqrt n)$ 
in probability\footnote{that is, the sequence of random variables $(N_n(a,b,c,d)-n \vol(a,b,c,d))/\sqrt{n}$ is tight for each $(a,b,c,d)\in I$.}. Since $\beta>1/2$, the error term for the rescaled process is $O(n^{1/2-\beta})=o(1)$. (Inequalities of Dvoretzky--Kiefer--Wolfowitz type show that these expansions hold \emph{uniformly} over all query rectangles $(a,b,c,d)$. 
Details will be given later.)
On the other hand, since $D_n^{(i)}$, $1\le i\le 4$, are bounded by the cost of executing a fully specified query retrieving some point, $\max_{1\le i\le 4} D_n^{(i)}$ is bounded above by the height of the tree $T_n$, and thus $O(\log n)$ in probability (see \cite{Devroye1987}, but much weaker bounds would also suffice). As a consequence, for any fixed $(a,b,c,d)\in I$, the limit of 
\[\frac{O_n(a,b,c,d) - n \vol(a,b,c,d)}{n^\beta}\] 
should be the limit of sums of terms of the form $n^{-\beta}Y_n^<$, $n^{-\beta}Y_n^\ge$, $n^{-\beta}\bar Y_n^{<}$ and $n^{-\beta}\bar Y_n^\ge $. Note that, of course, these terms are not independent; furthermore, the values for different $(a,b,c,d)\in I$ are dependent as well. We will prove that each of these terms converges uniformly on $I$ in probability. The four types of terms we now have to deal with are all ``partial-match''-like, and the technology we have developed in \cite{BrNeSu2011a,BrNeSu13} will come in handy.

\begin{rem} \label{rem:ext} \normalfont In \cite{ChDeZa2000}, the authors use the random variable $O_n^\Box(a,b,c,d)$ defined as the sum of $O_n(a,b,c,d)$ and the number of placeholders in $T_n$ whose associated region intersects the query rectangle $Q(a,b,c,d)$ as measure of complexity of the range query search algorithm. It is easy to see that, for all $(a,b,c,d) \in I$, we have
\begin{align*}O_n^\Box(a,b,c,d) = & ~  O_n(a,b,c,d)  + 3 N_n(a,b,c,d) + 1 \\
& + Y^\geq_n(b,d) -Y^\geq_n(b,c) + Y^<_n(a,d) -Y^<_n(a,c) \\
& + \bar Y ^\geq_n(d,b) - \bar Y^\geq_n(d,a) + \bar Y^<_n(c,b) - \bar Y^<_n(c,a). 
\end{align*}
Thanks to this identity, results similar to those presented in Section \ref{sec:results} can be formulated for the rescaled complexity 
$n^{-\beta}(O_n^\Box - 4 n \vol)$. We omit the details.
\end{rem}

\subsection{The contraction method} 

Most proofs in this work as well as in \cite{BrNeSu13} rely on the contraction method, a methodology which has proved fruitful in the probabilistic analysis of algorithms over the last thirty years. The contraction method was introduced by R{\"o}sler \cite{Ro91} in the analysis of the Quicksort algorithms and later developed by R{\"o}sler \cite{Roesler1992} and Rachev and R{\"u}schendorf \cite{RaRu1995}. We refer the reader  to \cite{RoRu01, NeRu04} for surveys on this topic. 

The main idea of the method is to set up a distributional or almost sure recurrence for the sequence of random variables under investigation and to derive a stochastic fixed-point equation for the corresponding limiting random variable. The proof of convergence relies on a contraction argument based either (i) on  metrics on suitable subspaces of probability distributions on the state space or (ii) on basic distances (often $L_p, p \geq 1$) if couplings of the sequence and the limit are available. The restriction to suitable subspaces typically requires mean or variance expansions of sufficient precision.

While most applications of the technique stay in the framework of real-valued random variables or vectors, extensions to the setting of continuous or c{\`a}dl{\`a}g functions covering both cases (i) and (ii) have been developed in the last few years; see \cite[Section 1]{NeSu} for an overview of the development of the method in general state spaces. In the context of this work, the most relevant literature is \cite{BrNeSu13, NeSu, BrSu15}.

\subsection{Plan of the paper}

The remainder of the paper is organized as follows: In Section~\ref{sec:partial_match}, we introduce the relevant background and constructions about partial match queries and their limit process that needs to be generalized to consider the terms involved in Lemma~\ref{lem:decomp}. Section~\ref{sec:one_sided} deals with the case of one-sided partial match queries, where only the points lying on one or the other side of the query line are reported (terms of the form $Y_n^<(t,1)$ and $Y_n^\ge(t,1)$). Section~\ref{sec:constrained} deals with the further restriction of the count of points depending on their second coordinate (along the direction parallel to the query line), namely terms of the form $Y_n^<(t,s)$ and $Y_n^\ge (t,s)$. We put everything together and complete the proofs in Section~\ref{sec:proofs}. Finally, the extensions to $2$-d trees are presented in Section~\ref{sec:kd}. 

\section{The cost of partial match queries in random quadtrees}
\label{sec:partial_match} 

In this section, we introduce the notation and review the relevant constructions and results about partial match queries that are central to our approach. We refer the reader to \cite{BrNeSu2011a,BrNeSu13} for more details and the proofs.

We call a set $Q \subseteq [0,1]^2$ a \emph{half-open} rectangle if, for some real numbers $a,b,c,d$ such that $0 < a < b \leq 1$ and $0 < c < d \leq 1$, we have
$$Q = \begin{cases} (a,b] \times (c,d], &  \text{or} \\ [0,b] \times (c,d], &  \text{or} \\ (a,b] \times [0,d], &   \text{or} \\ [0,b] \times [0,d]. & 
\end{cases}$$
Let $\mathcal Q$ be the set of all half-open rectangles and $\mathbb T = \bigcup_{ k \geq 0} \{1, 2, 3,4\}^k$ be the infinite quaternary tree of words on the alphabet $\{1,2,3,4\}$: the finite words on $\{1,2,3,4\}$ are the nodes, and the ancestors of some word are its prefixes (including the empty word $\varnothing$). For $u\in \mathbb T\setminus \{\varnothing \}$, the word $\bar u$ obtained by removing the last letter is called the parent of $u$. A subset $A\subset \mathbb T$ is called a tree if it is closed by taking prefixes. For a finite tree $A \subset \mathbb T$, let $\partial A$ be the set of nodes $u\in \mathbb T$ such that $u\not \in A$, but $\bar u \in A$. 

From the sequence $(X_i)_{i\ge 1}$ of random points in $[0,1]^2$, we recursively construct 
\begin{enumerate} \item [(i)] a bijection $\pi : \mathbb N \to \mathbb T$ inducing node labels $X_{\pi^{-1}(v)}$ associated with $v\in \mathbb T$, 
\item [(ii)] a family $\{Q_v \in \mathcal Q: v \in \mathbb T\}$ where $Q_\varnothing = [0,1]^2$, and for all $v\in \mathbb T$, $Q_v$ is the disjoint union of the four half-open rectangles $Q_{v1}, \ldots, Q_{v4}$, and
\item [(iii)] an increasing (for inclusion) sequence of trees $(T_n)_{n\ge 1}$ such that, for every $n\ge 1$, $\{Q_v : v \in \partial T_n\}$ is a partition of $[0,1]^2$.
\end{enumerate}
We proceed as follows. First, let $\pi(1) = \varnothing$, $T_1$ consist of the root node $\varnothing$ and $Q_1:= \SW(X_1)$, $Q_2 := \NW(X_1)$, $Q_3 := \SE(X_1)$ and $Q_4 := \NE(X_1)$
 be the four half-open rectangles generated by insertion of $X_1$ in $[0,1]^2$. Next, for $n \geq 1$, having defined $\pi(j)$ for all $j  \leq n$, the tree $T_n$ and rectangles $Q_v$ for all $v \in \partial T_n$ such that $\{Q_v : v \in \partial T_n\}$ is a partition of $[0,1]^2$, we let $\pi(n+1)$ be the unique node $v \in \partial T_n$ with $X_{n+1} \in Q_v$. Further, $T_{n+1} := T_n \cup \{v\}$ and $Q_{v1}, \ldots, Q_{v4}$ are defined respectively as $Q_v\cap \SW(X_{n+1})$, $Q_v\cap \NW(X_{n+1})$, $Q_v\cap \SE(X_{n+1})$ and $Q_v\cap \NE(X_{n+1})$. The partition of the unit square induced by $T_{n+1}$ is then given by $\{Q_v : v \in \partial T_{n+1}\}$. 
 The sets $(\mathcal  R_i)_{i \geq 1}$ from definition \eqref{def:On}  can be chosen as
 \begin{align} \label{defRformal}
 \mathcal  R_i = Q_{\pi(i)}, \quad i \geq 1.
 \end{align} 
Let $Q_v^{(i)}, i=1,2$ be the projection of $Q_v$ on the $i$th component. For $v \in \mathbb T$, we define the time-transformations $\varphi_v, \varphi_v'$ quantifying the position of a point $(t,s)$ relative to the boundary of $Q_v$: we set
\begin{align} \label{phi1} \varphi_v(t) =  \mathbf 1_{Q_{v}^{(1)}}(t) \frac{t - \inf Q_{v}^{(1)}}{\sup Q_{v}^{(1)} - \inf Q_{v}^{(1)}}, \quad t\in [0,1],  \end{align} and \begin{align} \label{phi2} \varphi_v'(s) =  \mathbf 1_{Q_{v}^{(2)}}(s) \frac{s - \inf Q_{v}^{(2)}}{\sup Q_{v}^{(2)} - \inf Q_{v}^{(2)}}, \quad s\in [0,1]. \end{align}
With $X_{\pi^{-1}(v)} = (x_1^v, x_2^v)$, we set
$U^v := \varphi_v(x_1^v)$ and  $V^v := \varphi_v'(x_2^v).$
By construction, $\{(U^v, V^v): v \in \mathbb T \}$ is a family of independent random variables with the uniform distribution on $[0,1]^2$. To keep the notation simple, we write $U := U^\varnothing$ and $V := V^\varnothing$; more generally, we usually drop the reference to $\varnothing$ when the meaning is clear from the context.

\medskip 
\noindent\textbf{The limit process $\cal Z$.}
In the remainder of the manuscript, we write $\Ck, k = 1,2$ for the space of continuous functions on $[0,1]^k$ endowed with the supremum norm $\| f \| = \sup_{t \in [0,1]^k} |f(t)|$.
Recall the constants $K_1, \beta$ and the function $h$ from equations \eqref{limit_unif_const} and \eqref{limit_t_const}. For $v\in \mathbb T$, let $|v|$ denote the length of the word (the distance between $v$ and the root $\varnothing$). Further, for $Q \in \mathcal Q$, let $|Q|$ denote its area. Set $\mathcal Z_0^v = K_1 h$ for all $v \in \mathbb T$, and, recursively, 
\begin{align} \label{const:m}
\mathcal Z_{n+1}^v(t) = \sum_{r=1}^4 \mathbf 1_{Q_{vr}^{(1)}}(t) A_{vr}^\beta \mathcal Z_n^{vr}(\varphi_{vr}(t)), \quad v \in \mathbb T,  
\end{align}
where $A_{v} = |Q_{v}| / |Q_{\bar v}|$ for $v \in \mathbb T, v \neq \varnothing$. In other words, for all $v \in \mathbb T$, we have
$$A_{v1} = U^v V^v, \quad A_{v2} = U^v(1-V^v), \quad A_{v3} = (1-U^v) V^v, \quad \text{and } A_{v4} = (1-U^v) (1-V^v).$$
By the results stated in Proposition 2, Theorem 5, Proposition 9, Lemma 10 and Proposition 11  in \cite{BrNeSu13}, there exist random continuous functions $\mathcal Z^v, v \in \mathbb T$, such that
\begin{enumerate}
\item [(i)] the random variables $\mathcal Z^v, v \in \mathbb T,$ are identically distributed,
\item [(ii)] $\|\mathcal Z_n^v - \mathcal Z^v \| \to 0$ almost surely and with convergence of all moments,
\item [(iii)] $\Ec{\mathcal Z^v(t)^m} = c_m h(t)^m$ for appropriate constants $c_m > 0$ where $c_1 = K_1$, 
\item [(iv)] $\Ec{\|\mathcal Z^v\|^p} <  \infty$ for all $p > 0$, 
\item [(v)] $\mathcal Z^{v1}, \ldots, \mathcal Z^{v4}, U^v, V^v $ are stochastically independent and, almost surely, for all $t \in [0,1]$,
\begin{align*} \mathcal Z^v(t) = \sum_{r=1}^4 \mathbf 1_{Q_{vr}^{(1)}}(t) A_{vr}^\beta \mathcal Z^{vr}(\varphi_{vr}(t)), \end{align*} 
\item [(vi)] up to a multiplicative constant, $\mathcal Z^v$ is the unique continuous process (in distribution) with $\Ec{\|\mathcal Z^v\|^2} < \infty$ satisfying the stochastic fixed-point equation
\begin{align} \label{fix:lim} \mathcal Z^v \stackrel{d}{=} \left(\sum_{r=1}^4 \mathbf 1_{Q_{r}^{(1)}}(t) A_{r}^\beta    \mathcal Z^{(r)}(\varphi_{r}(t))\right)_{t \in [0,1]}. \end{align} 
Here, $\mathcal Z^{(1)}, \ldots, \mathcal Z^{(4)}$ are copies of $\mathcal Z^v$, and $\mathcal Z^{(1)}, \ldots, \mathcal Z^{(4)}, U, V$ are independent.
\end{enumerate}
\begin{figure}[htbp]
	\centering
    \begin{picture}(200,200)
	\put(-20,0){\includegraphics[width=0.6\textwidth]{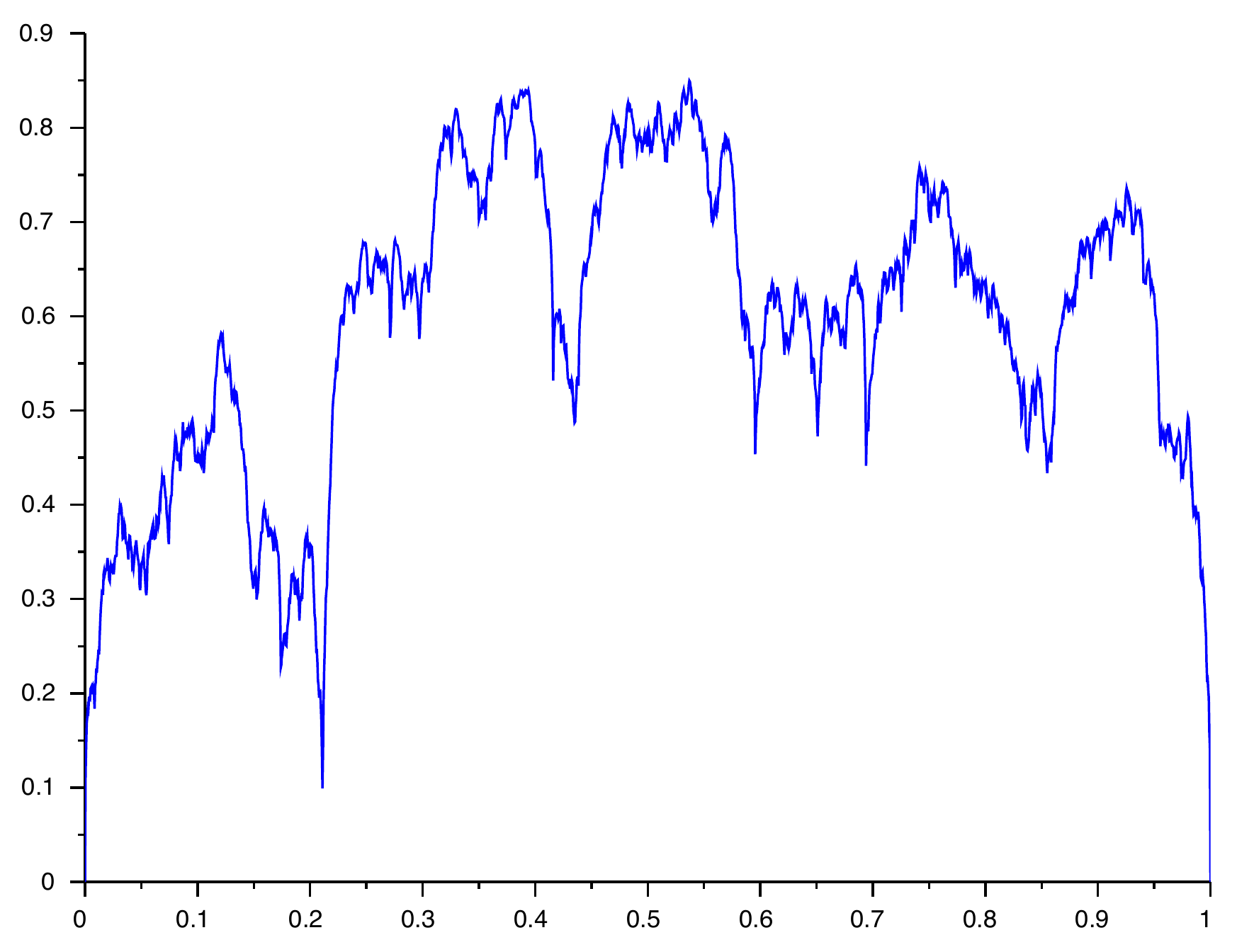}}
    \put(-10,200){${\cal Z}(s)$}
    \put(235,10){$s$}
    \end{picture}
	\caption{A simulation of the limit process ${\cal Z}=({\cal Z}(s))_{s\in [0,1]}$ for the cost 
    of partial match queries in \eqref{conv_st}.}
	\label{fig:Z-process}
\end{figure} 

\medskip 
\noindent\textbf{The complexity of partial match queries.}
The main result in \cite{BrNeSu13} is the limit law \eqref{lim_fun} with limit process $\mathcal Z := \mathcal Z^\varnothing$. The identity \eqref{fix:lim} for the distribution of $\cal Z$ is reminiscent of the following distributional recurrence for the discrete process $C_n$: letting $N_1, \ldots, N_4$ denote the subtree sizes of the quadtree $T_n$ (such that $N_1 + \cdots + N_4 = n-1$), we have
\begin{align} 
\label{rec:pm} 
C_n \stackrel{d}{=} \left( \sum_{r=1}^4 \mathbf 1_{Q_r^{(1)}}(t)  C^{(r)}_{N_r}(\varphi_r(t)) + 1 \right)_{t \in [0,1]}.
\end{align}
Here, the random sequences $(C_n^{(1)})_{n \geq 0}, \ldots, (C^{(4)}_n)_{n \geq 0}$ stemming from the complexities in the four subtrees are independent and identically distributed and also independent from $N_1, \ldots, N_4, U, V$. Further, given $(U, V)$, the vector $(N_1, \ldots, N_4)$ has the multinomial distribution with parameters $(n-1;
A_1, \ldots, A_4)$.
This recurrence is at the heart of the proof of \eqref{lim_fun}, and similar recurrences also play key roles in the present work. 
As already mentioned in the introduction, by \cite[Corollary 1.2]{cucpc}, in probability (and therefore with convergence of all moments by \cite[Theorem 4]{BrNeSu13}), 
\begin{align}\label{conv_st}
\|n^{-\beta} C_n - \mathcal Z \| \xrightarrow[n\to\infty]{} 0.
\end{align}

\section{One-sided partial match queries}
\label{sec:one_sided}
 
For $t \in [0,1]$, let $C_n^{\geq}(t)$ denote the number of nodes with first coordinate at least $t$ visited by a partial match query retrieving all keys with first component $t$ in $T_n$. 
In the partition of the unit square induced by $T_n$, $C_n^\geq(t)$ can be identified as the number of horizontal lines intersecting the vertical line at $t$ stemming from the points lying to the right of $t$. Again, we agree $C_n^\geq(\,\cdot\,)$ to be a right-continuous step function. Set $C_n^<(t) = C_n(t)  - C_n^{\geq}(t)$. By construction, $C_n^\geq$ satisfies a recurrence very similar to \eqref{rec:pm}, namely
\begin{align} \label{rec:onesided} (C_n^\geq(t))_{t \in [0,1]} \stackrel{d}{=} \left( \sum_{r=1}^4 \mathbf 1_{Q_r^{(1)}}(t)  C^{\geq, (r)}_{N_r}(\varphi_r(t))  +  \mathbf 1_{[0,U)}(t) \right)_{t \in [0,1]},\end{align}
with conditions on independence and distributions as in  \eqref{rec:pm}. Since the only difference between \eqref{rec:pm} and \eqref{rec:onesided} concerns the additive term which asymptotically vanishes after rescaling by $n^{\beta}$, one might guess that $C_n^\geq$ admits the same (distributional) scaling limit as $C_n$ modulo a multiplicative constant. In fact, Proposition~\ref{prop:onesidedpm} below shows much more. The following lemma is an important ingredient of the proof.
 
\begin{lem}\label{lem:mean}
There exists $\varepsilon > 0$ such that, uniformly in $t \in [0,1]$ and as $n \to \infty$, 
\begin{align} \label{ex_meone}
\Ec{C^\geq_n(t)} = \frac {K_1} 2 h(t)  n^\beta + O(n^{\beta - \varepsilon}), 
\qquad \text{and}\qquad 
\Ec{C^<_n(t)} = \frac {K_1} 2 h(t)  n^\beta + O(n^{\beta - \varepsilon}). 
\end{align}
Here, $K_1$ and $\beta$ are the constants given in \eqref{limit_unif_const} and \eqref{limit_t_const}.
\end{lem}
The proof of Lemma~\ref{lem:mean} is postponed until the end of the section.
 
\begin{proposition} \label{prop:onesidedpm}
In probability and with convergence of all moments, as $n \to \infty$, 
$$\|n^{-\beta} C_n^{\geq} - \mathcal Z /2\| \to 0, 
\qquad \text{and}\qquad 
\|n^{-\beta}  C_n^{<} - \mathcal Z/2 \| \to 0.$$
\end{proposition}
\begin{proof}
Since $\|n^{-\beta}C_n -{\cal Z}\|\to 0$ by \eqref{conv_st}, it suffices to show that $n^{-\beta} \|C_n^{\geq} - C_n^{<}\|  \to 0$ in probability and with respect to all moments. To this end, write $\mu^\geq_n(t) = \Ec{C_n^{\geq}(t) }$, $\mu^<_n(t) = \Ec{ C_n^<(t) }$ and let 
\[
W_n(t) : = \frac{C_n^{\geq}(t) - \mu^\geq_n(t)}{n^\beta} -  \frac{C_n^{<}(t) -\mu^<_n(t))}{n^\beta}.\] 
Note that $\Ec{W_n(t)} = 0$ for all $t \in [0,1]$.
By construction, the process $W_n$ satisfies the following functional distributional recurrence:
\begin{align} \label{add_term}
W_n & \stackrel{d}{=} 
\Bigg( \sum_{r=1}^4 \mathbf 1_{Q_r^{(1)}}(t)  \left(\frac{N_r}{n} \right)^\beta W^{(r)}_{N_r}(\varphi_r(t)) \\
& \qquad + \frac{
\sum_{r=1}^4 \mathbf 1_{Q_r^{(1)}}(t) \big[\mu^\geq_{N_r}(\varphi_r(t)) - \mu^<_{N_r}(\varphi_r(t))\big]  - \mu^\geq_n(t)+\mu^<_n(t) + \mathbf 1_{[U,1]}(t) - \mathbf 1_{[0,U)}(t) }{n^\beta} \Bigg)_{t \in [0,1]}, \nonumber 
\end{align}
again with assumptions on independence and distributions as in \eqref{rec:pm}. 

By Lemma \ref{lem:mean}, uniformly in $t \in [0,1]$,  the additive term \eqref{add_term} converges to zero almost surely and with respect to all moments. Therefore, and as for every $r\in\{1,2,3,4\}$, $N_r / n \to A_r$ almost surely by the concentration of the binomial distribution, one would expect that, if $W_n$ admits a limit process $W$, then $W$ should satisfy the distributional fixed-point equation obtained by taking limits in the recurrence above:
\begin{align}\label{0proc}
W \stackrel{d}{=} \left(\sum_{r=1}^4 \mathbf 1_{Q_r^{(1)}}(t)  A_r^\beta W^{(r)}(\varphi_r(t))\right)_{t \in [0,1]}.
\end{align}
Here, $W^{(1)}, \ldots, W^{(4)}$ are copies of $W$, and $W^{(1)}, \ldots, W^{(4)}, U, V$ are stochastically independent ($U$ and $V$ appear in the definition of $A_r$, $r=1,\dots, 4$).
As this equation is homogeneous, it is solved by the process which is identical to zero. Furthermore, an application of \cite[Lemma 18]{NeSu} shows that the zero process is the unique solution (in distribution) of \eqref{0proc} among all random processes with zero mean and finite absolute second moment. As the proof of \eqref{lim_fun} in \cite{BrNeSu13}, the rigorous verification that the convergence $\|W_n\| \to 0$ holds in probability makes use of the functional contraction method. The application of \cite[Theorem 22]{NeSu} requires to verify a set of conditions (C1)--(C5) formulated in that paper. By the similarity of the processes $C_n$ and $C_n^\geq$ and their distributional recurrences, conditions (C1), (C2), (C4) and (C5) can be verified exactly in the same way as it was done in \cite{BrNeSu13} for the process $C_n$, and we omit the details. The fact that the zero process is a solution of \eqref{0proc} guarantees  (C3). This shows distributional (or, equivalently, stochastic) convergence. The convergence of moments follows by monotonicity since $\max\{C_n^{<},C_n^\ge\} \le C_n$ and $\sup_{n \geq 1} n^{-\beta p} \Ec{ \|C_n\|^p} <  \infty$ by \cite[Theorem 4]{BrNeSu13}.
\end{proof}

\begin{cor} \label{cor:br}
In probability and with respect to all moments, we have
\[\sup_{0\le s,t \le 1} 
\left|\frac{O_n(s,t,0,1) - n(t-s)}{n^\beta} - \frac{\mathcal Z(t)+\mathcal Z(s)}2 \right| 
\xrightarrow[n\to\infty]{} 0. \]
\end{cor}

\begin{proof}
Let $N_n(t)$ denote the total number of points among $X_1,X_2,\dots, X_n$ that have a first coordinate at least $t$. By Donsker's classical theorem for empirical distribution functions, we have
\begin{align}\label{bridge} 
\left(\frac{N_n(t)  -n(1-t)}{\sqrt n}\right)_{t \in [0,1]} \xrightarrow[n\to\infty]{d}B,
\end{align}
where $B$ is a Brownian bridge, that is, $B(t) = W(t) - tW(1), t \in [0,1] $ for a standard Brownian motion $W$. It is well-known that this convergence is also with respect to all moments (this follows, e.g., from the Dvoretzky--Kiefer--Wolfowitz inequality \cite{DvKiWo1956a}).
Thus, since $\beta>1/2$, uniformly for $s,t\in [0,1]$, we have $n^{-\beta} (N_n(t) - N_n(s)) \to 0$ in probability and with respect to all moments. As $O_n(s,t,0,1) = N_n(t) - N_n(s) + C_n^\geq(t) + C_n^<(s)$, the assertion follows from Proposition~\ref{prop:onesidedpm}. 
\end{proof}

Finally, we prove Lemma~\ref{lem:mean} that was instrumental in the proof of Proposition~\ref{prop:onesidedpm}.

\begin{proof}[Proof of Lemma \ref{lem:mean}]
The most technical ingredient in the proof of \eqref{lim_fun} that appears in \cite{BrNeSu13} is the following strengthening of \eqref{limit_t}: there exists $\varepsilon > 0$ such that, uniformly in $t \in [0,1]$, and as $n \to \infty$
\begin{align}\label{ex_me} 
\Ec{C_n(t)} = K_1 h(t)  n^\beta + O(n^{\beta - \varepsilon}). 
\end{align}
This result heavily relies on the methods developed by Curien and Joseph in \cite{CuJo2010}. Note that the fact that the first point $X_1$ falls on one side of the line $\{x=t\}$ induces an asymmetry between $C_n^{<}(t)$ and $C_n^{\ge}(t)$ for $n$ fixed, and the means $\Ec{C_n^{<}(t)}$ and $\Ec{C_n^{\ge }(t)}$ are different (unless $t=1/2$). Somewhat as a consequence of this inherent asymmetry, we have no simple/soft argument to deduce \eqref{ex_meone} directly from \eqref{ex_me}, and it is necessary to repeat all steps in the verification of the latter. To work out all details would go beyond the scope of this note, and we confine ourselves to the  discussion of the main steps: first of all, as in \cite[Section 5]{BrNeSu13}, one considers a Poissonized model in which the points $X_1, X_2, \ldots$ are inserted following the arrival times of a homogeneous Poisson process on $[0, \infty)$. Consequently, we deal with a family of increasing quadtrees $(T_s)_{s \ge 0}$ built on the points arrived before time $s$ (there are a Poisson$(s)$ number of such points), and the corresponding partial match query complexities $(C_s(t))_{s\ge 0}$ and $(C_s^\geq(t))_{s\ge 0}$. Standard de-poissonization arguments based on the concentration of the Poisson distribution imply that it is sufficient to show the corresponding statements for the continuous-time process, namely there exists $\varepsilon > 0$ such that
$$\sup_{0\le t \le 1} \left| s^{-\beta} \Ec{C_s^\geq(t)} - \frac{K_1}2 h(t)\right| = O(s^{-\varepsilon}), \quad s\to \infty.$$
Following \cite{BrNeSu13}, one distinguishes between the behavior at the boundary ($t \in [0, \delta] \cup [1-\delta, 1]$ for small $\delta>0$) and away from the boundary ($t \in [\delta, 1-\delta]$). In \cite{BrNeSu13}, it was enough to consider $t \leq 1/2$ by symmetry, but here, this is not the case. Lemmas 14 and 15 in \cite{BrNeSu13} contain the corresponding bounds for the process $C_s$, and we now argue that these bounds apply with the \emph{same} involved constants to the one-sided quantity $C_s^\geq$.

Concerning the behavior at the boundary, we have the trivial bound
\begin{align*}
\sup_{t \in [0, \delta] \cup [1-\delta,1]} \left |s^{-\beta}\Ec{C_s^\geq(t)} - \frac{K_1} 2 h(t)  \right|
& \leq \sup_{t \in [0, \delta] \cup [1-\delta,1]} s^{-\beta}\Ec{C_s^\geq(t)} + \sup_{t \in [0, \delta] \cup [1-\delta,1]} K_1 h(t) \\
& \leq  \sup_{t \in [0, \delta]} s^{-\beta}\Ec{C_s(t)} + \sup_{t \in [0, \delta]} \frac{K_1} 2 h(t)\,, \end{align*}
which brings us back to the situation of the two-sided problem and explains why the bound in \cite[Lemma 14]{BrNeSu13} also applies in our case. The main part of the proof of \eqref{ex_me} is contained in \cite[Lemma 14]{BrNeSu13} and relies on a coupling argument between $\varphi_v(t)$ for the unique node $v=v_1v_2\ldots v_k, k \geq 1,$ with $v_i \in \{1,3\}$ for all $i = 1, \ldots, k$ and $t \in Q_v$ and a uniformly distributed random variable $\xi$. We do not explain this step in detail but mention that one distinguishes two cases: first, in case 1 (coupling has not yet happened), similarly to the last display, one uses the crucial uniform upper bound in \cite[Lemma 2]{CuJo2010}:
$$\sup_{s \geq 0} \sup_{t \in [0,1]} s^{-\beta} \Ec{C_s(t)} < \infty.$$
By monotonicity, this bound can also be applied to $C_s^\geq$. In case 2 (coupling has happened), the main ingredient in the proof is the expansion in \eqref{limit_unif_const}. (Here, the second order term is important.) Since for a uniform random variable $\xi$ independent of $(X_i)_{i\ge 1}$, we have $\Ec{C_n^\geq(\xi)} = \Ec{C_n^<(\xi)} =  \E{C_n(\xi)}/ 2$, the same arguments apply in our case. 
\end{proof}

\section{Constrained partial match queries}
\label{sec:constrained}

\subsection{Preliminary considerations}

For $t,s \in [0,1]$, let $Y_n(t,s) = O_n(t,t,0,s)$ be the number of nodes with second coordinate at most $s$ visited by the partial match query retrieving the points with first coordinate equal to $t$. By our convention, $Y_n(t,s)$ is right continuous in both coordinates. Note that $Y_n(t,1) = C_n(t)$ for $t \in [0,1]$. To prove a functional limit theorem for $Y_n$, we use a variant of the functional contraction method which makes explicit use of our encoding. A very similar approach was taken in \cite{BrSu15} when studying the dual tree of a partitioning of the disk by sequential insertions of non-crossing random chords. As a result, we provide a proof of convergence of $n^{-\beta}Y_n$ that avoids another application of the complex machinery developed in \cite{NeSu}. 

Recalling the time-transformations \eqref{phi1} and \eqref{phi2}, we have the  following distributional recursive equation on the level of random fields with parameter space $[0,1]^2$:
\begin{equation} \label{recy}
\begin{aligned}
 \big(Y_n(t,s)\big)_{t,s\in [0,1]} & \stackrel{d}{=} \Big( \mathbf 1_{Q_1} (t,s) Y^{(1)}_{N_1}(\varphi_1(t), \varphi'_1(s)) \\
 & \qquad + \mathbf 1_{Q_2} (t,s) \left[ Y^{(1)}_{N_1}(\varphi_1(t), 1) + Y^{(2)}_{N_2}(\varphi_2(t), \varphi'_2(s)) \right] + \mathbf 1_{[V,1]}(s) \\
& \qquad + \mathbf 1_{Q_3} (t,s) Y^{(3)}_{N_3}(\varphi_3(t), \varphi'_3(s)) \\
&\qquad + \mathbf 1_{Q_4} (t,s) \left[ Y^{(3)}_{N_3}(\varphi_3(t), 1) + Y^{(4)}_{N_4}(\varphi_4(t), \varphi'_4(s)) \right] \Big)_{t,s \in [0,1]}.
\end{aligned}
\end{equation}
Here, we have the same conditions on independence and distributions as in \eqref{rec:pm}.
Thus, if $n^{-\beta} Y_n$ converges, we expect the distribution of the limit $Y$ to satisfy the following fixed-point equation
\begin{equation*}
\begin{aligned}
\big(Y(t,s)\big)_{t,s\in [0,1]} & \stackrel{d}{=} 
\Big( \mathbf 1_{Q_1} (t,s) A_1^\beta Y^{(1)}(\varphi_1(t), \varphi'_1(s)) \\
&\qquad + \mathbf 1_{Q_2} (t,s) \left[ A_1^\beta  Y^{(1)}(\varphi_1(t), 1) + A_2^\beta Y^{(2)}(\varphi_2(t), \varphi'_2(s)) \right]  \\
&\qquad+ \mathbf 1_{Q_3} (t,s) A_3^\beta Y^{(3)}(\varphi_3(t), \varphi'_3(s)) \\
&\qquad+ \mathbf 1_{Q_4} (t,s) \left[A_3^\beta  Y^{(3)}(\varphi_3(t), 1) + A_4^\beta  Y^{(4)}(\varphi_4(t), \varphi'_4(s)) \right] \Big)_{t,s \in [0,1]},
\end{aligned}
\end{equation*}
where $Y^{(1)}, \ldots, Y^{(4)}$ are copies of $Y$, and the random variables $Y^{(1)}, \ldots, Y^{(4)}, U, V$ are independent. 
Since $Y_n(t,1) = C_n(t)$, and this process is well understood, the crucial observation is that, for any fixed $(t,s) \in [0,1]^2$, only \emph{one} of the processes $(Y^{(1)}_n), \ldots, (Y_n^{(4)})$  contributes to the recursive decomposition \eqref{recy} at a point whose second coordinate differs from one. The same can be said about the associated stochastic fixed-point equation. It is this fact why we do not need to engage the methodology of \cite{NeSu} to show convergence. (We do however need some ideas of this work to characterize the distribution of $Y$. See Proposition \ref{prop:idY} below.)

\subsection{Construction of the limit process and convergence}

We proceed as in the construction of the process $\mathcal Z$ described in Section~\ref{sec:partial_match}. To simplify the notation, let us introduce the following operators: for $v \in \mathbb T$, and for a function $f:[0,1]^2\to \mathbb R$, define
\begin{equation}\label{opB}
\begin{aligned}
B^v_1(f)(t,s) & = A_{v1}^\beta \left[ \mathbf 1_{Q_{v1}} (t,s) f(\varphi_{v1}(t), \varphi'_{v1}(s)) + \mathbf 1_{Q_{v2}} (t,s) f(\varphi_{v1}(t), 1)\right], \\
B^v_2(f)(t,s) & =  A_{v2}^\beta \mathbf 1_{Q_{v2}} (t,s) f(\varphi_{v2}(t), \varphi'_{v2}(s)), \\
B^v_3(f)(t,s) & = A_{v3}^\beta \left[ \mathbf 1_{Q_{v3}} (t,s)    f(\varphi_{v3}(t), \varphi'_{v3}(s)) + \mathbf 1_{Q_{v4}} (t,s) f(\varphi_{v3}(t), 1)\right], \\
B^v_4(f)(t,s) & =  A_{v4}^\beta \mathbf 1_{Q_{v4}} (t,s) f(\varphi_{v4}(t), \varphi'_{v4}(s)). 
\end{aligned}
\end{equation}
For all $v \in \mathbb T$, let $\mathcal Y_0^v(t,s) = K_1 h(t)$ for all $t,s\in [0,1]$. Then, recursively, set
\begin{align}\label{def:y} 
\mathcal  Y_{n+1}^v(t,s) = \sum_{r=1}^4 B^v_r(\mathcal Y_{n}^{vr})(t,s), \qquad v \in \mathbb T.
\end{align}
This definition extends the construction of $\mathcal Z^v_n$ in \eqref{const:m} since we have $\mathcal Y_n^v(t,1) = \mathcal Z_n^v(t)$ for $t \in [0,1]$. We first verify that this indeed allows to construct a family of processes $({\cal Y}^v)_{v\in \mathbb T}$ that have the required properties:

\begin{proposition} \label{prop:limity}
There exist random continuous $\Ctwo$-valued fields $\mathcal Y^v, v \in \mathbb T$, such that
\begin{compactenum}[(i)]
\item the random variables $\mathcal Y^v, v \in \mathbb T,$ are identically distributed,
\item $\|\mathcal Y_n^v - \mathcal Y^v\| \to 0$ almost surely and with convergence of all moments,
\item $\mathcal Y^v(t,1) = \mathcal Z^v(t)$ for all $t \in [0,1]$,
\item $\Ec{ \|\mathcal Y^v\|^p} <  \infty$ for all $p > 0$, and
\item $\mathcal Y^{v1}, \ldots, \mathcal Y^{v4}, U^v, V^v $ are stochastically independent and, almost surely, for all $t,s \in [0,1]$,
\begin{align} \label{fix:lim2} \mathcal Y^v(t,s) = \sum_{r=1}^4 B^v_r(\mathcal Y^{vr})(t,s),
\end{align} 
where the operators $B_r^v$, $r\in \{1,2,3,4\}$, $v\in \mathbb T$, are defined in \eqref{opB}.
\end{compactenum}
\end{proposition}
\begin{proof}
By the definition in \eqref{def:y} of the family of process $({\cal Y}^v_n)_{n\ge 0}$, $v\in \mathbb T$, we have 
\begin{align*} 
& [\mathcal Y_{n+1}^v(t,s) - \mathcal Y_{n}^v(t,s)]^2 \\
& = \sum_{r=1}^4 \mathbf 1_{Q_{vr}} (t,s) A_{vr}^{2\beta} [\mathcal Y_{n}^{vr}(\varphi_{vr}(t),\varphi_{vr}'(s)) - \mathcal Y_{n-1}^{vr}(\varphi_{vr}(t),\varphi_{vr}'(s))]^2 \\
& + \mathbf 1_{Q_{v2}}(t,s) A_{v1}^{2\beta} [\mathcal Y_{n}^{v1}(\varphi_{v1}(t),1) - \mathcal Y_{n-1}^{v1}(\varphi_{v1}(t),1)]^2 \\ 
& +  \mathbf 1_{Q_{v4}}(t,s) A_{v3}^{2\beta} [\mathcal Y_{n}^{v3}(\varphi_{v3}(t),1) - \mathcal Y_{n-1}^{v3}(\varphi_{v3}(t),1)]^2 \\
& + 2 \mathbf 1_{Q_{v2}}(t,s) A_{v1}^\beta A_{v2}^\beta [\mathcal Y_{n}^{v1}(\varphi_1(t),1) - \mathcal Y_{n-1}^{vr}(\varphi_1(t),1)]\cdot [\mathcal Y_{n}^{v2}(\varphi_2(t),\varphi_2'(s)) - \mathcal Y_{n-1}^{v2}(\varphi_2(t),\varphi_2'(s))] \\
& +  2 \mathbf 1_{Q_{v4}}(t,s) A_{v3}^\beta A_{v4}^\beta [\mathcal Y_{n}^{v3}(\varphi_3(t),1) - \mathcal Y_{n-1}^{v3}(\varphi_3(t),1)]\cdot [\mathcal Y_{n}^{v4}(\varphi_4(t),\varphi_4'(s)) - \mathcal Y_{n-1}^{v4}(\varphi_4(t),\varphi_4'(s))].
\end{align*} 
Let $\Delta_n = \Ec{\|\mathcal Y_{n+1}^v - \mathcal Y_{n}^v \|^2}$ for $n \geq 0$. First taking the expectation of the supremum over $(t,s) \in [0,1]^2$ on either side and then applying the Cauchy--Schwarz inequality on the right hand side shows that, for $n \geq 1$,
$$\Delta_n \leq \Delta_{n-1} \sum_{r=1}^4 \Ec{A_{vr}^{2\beta}} + 4 \sqrt{\Delta_{n-1}\cdot \mathbb E\bigg[\sup_{t \in [0,1]} |\mathcal Y_{n}^v(t,1) - \mathcal Y_{n-1}^v(t,1)|^2\bigg]}.$$
Now, by \cite[Proposition 9]{BrNeSu13}, there exist constants $C > 0$ and $q \in (0,1)$ such that, 
for every $n\ge 1$, 
\begin{align} \label{exp:b} 
\mathbb E\bigg[\sup_{t \in [0,1]} |\mathcal Y_{n}^v(t,1) - \mathcal Y_{n-1}^v(t,1)|^2\bigg]  = \Ec{\|\mathcal Z_n^v - \mathcal Z_{n-1}^v \|^2}\leq C^2 q^n\,.
\end{align}  
As a consequence, setting 
\begin{align}\label{gamma} 
\gamma := \sum_{r=1}^4 \Ec{A_{r}^{2\beta}} = \frac 4 {(2\beta+1)^2}< 1, 
\end{align} we obtain 
$$\Delta_n \leq \gamma \Delta_{n-1} + 4 C q^{n/2} \sqrt{\Delta_{n-1}}.$$
By induction on $n\ge 0$ it follows that $\Delta_n = O(r^n)$ for all $\sqrt q < r < 1$. (This argument is worked out in the proof of
\cite[Lemma 2.3]{BrSu15}.) Uniform almost sure convergence follows straightforwardly from the completeness of $\Ctwo$. (These details are explained in the proof of \cite[Theorem 1.2] {BrSu15}.) Convergence of the $p$th moment follows by induction on $p$ along the same lines. (Here, one uses that an exponential bound of the form \eqref{exp:b} is valid for any higher moment.) Alternatively, one can prove boundedness of the sequence $\Ec{ \| Y_n \|^p},  n \geq 1$ by induction on $p \geq 2$ using \eqref{def:y} by induction on $n \geq 0$.
\end{proof}
\begin{figure}[tb]
    \centering
    \includegraphics[width=0.45\textwidth]{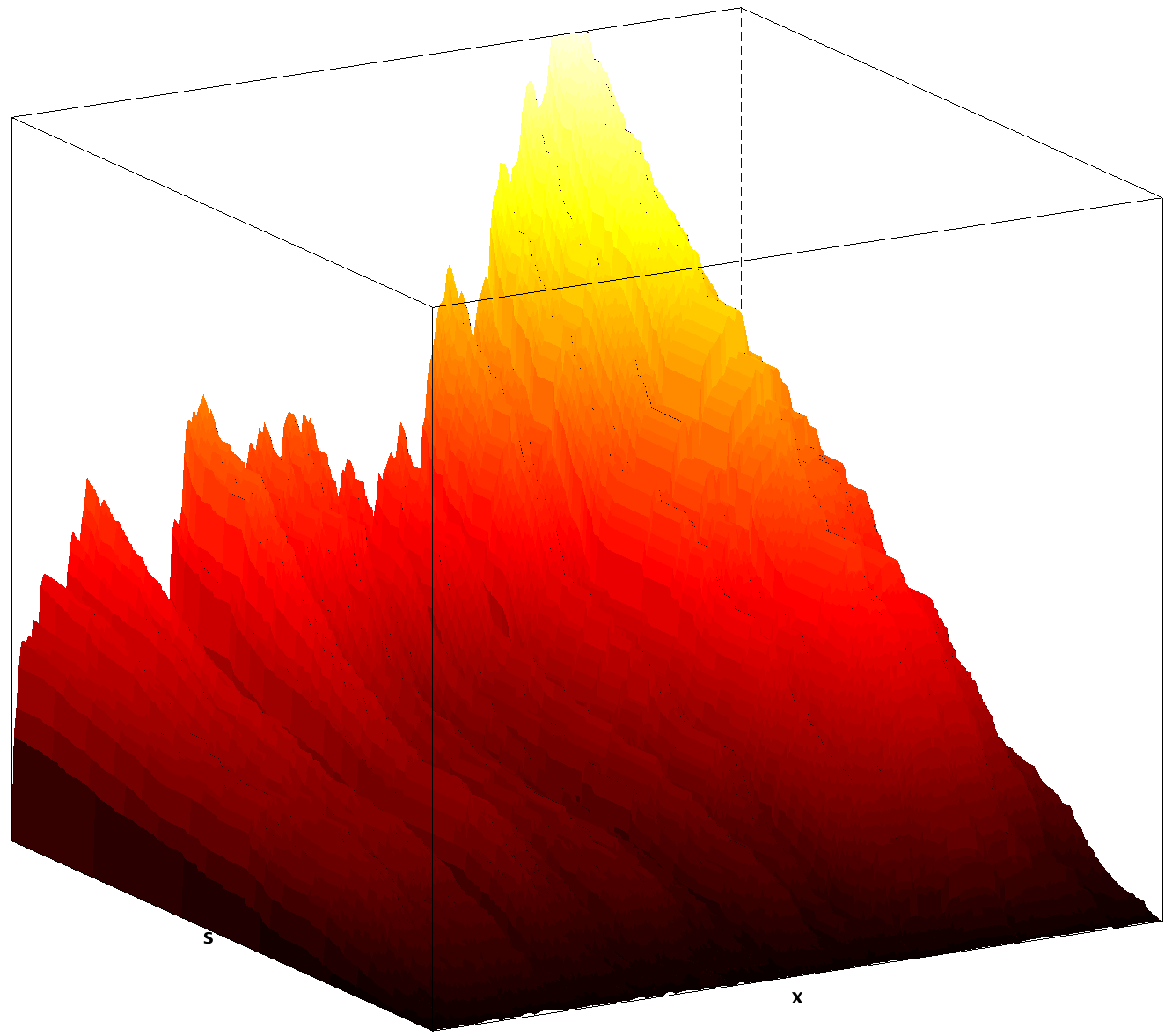}\hfil
    \includegraphics[width=0.45\textwidth]{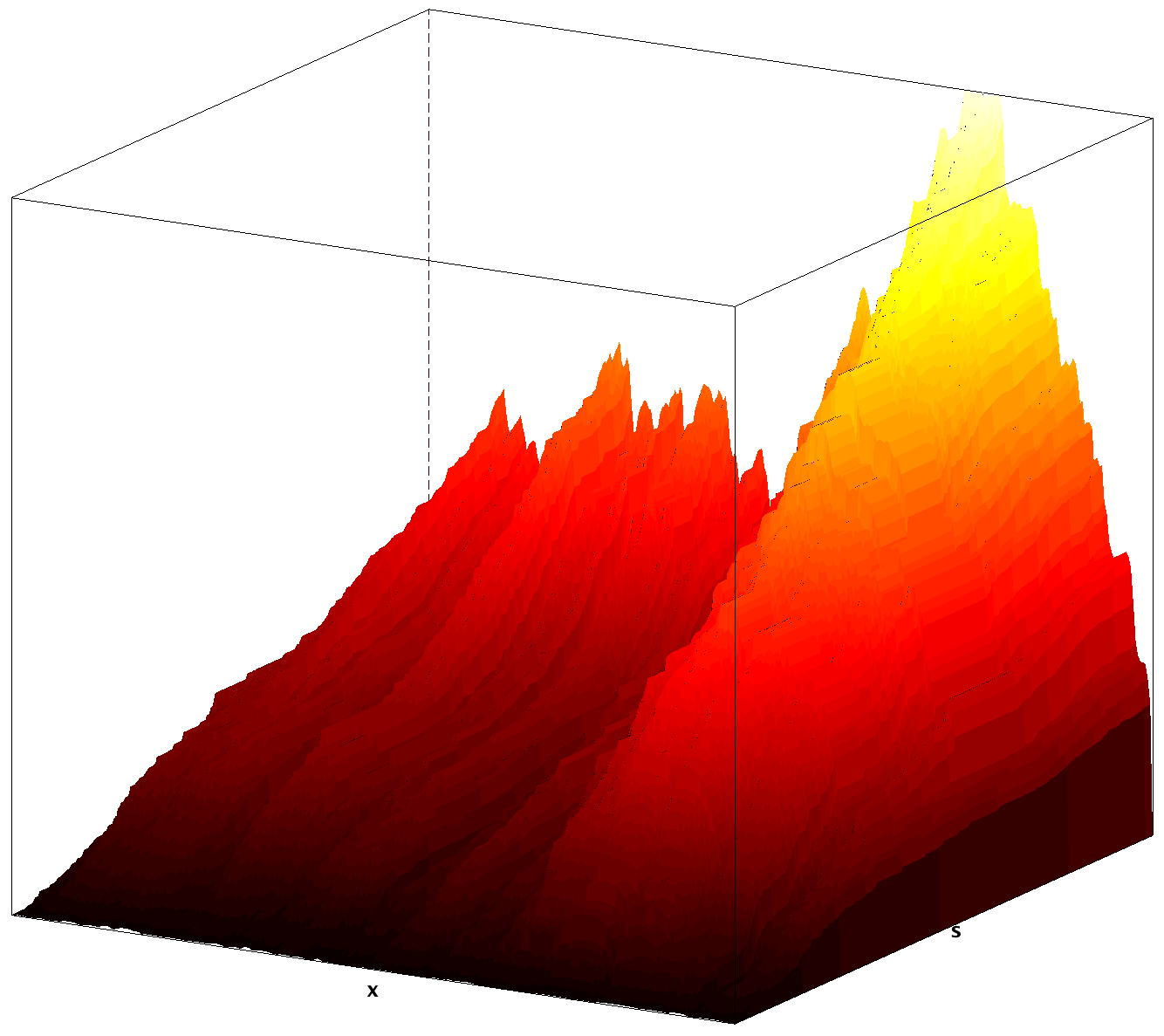}   
    \includegraphics[width=0.45\textwidth]{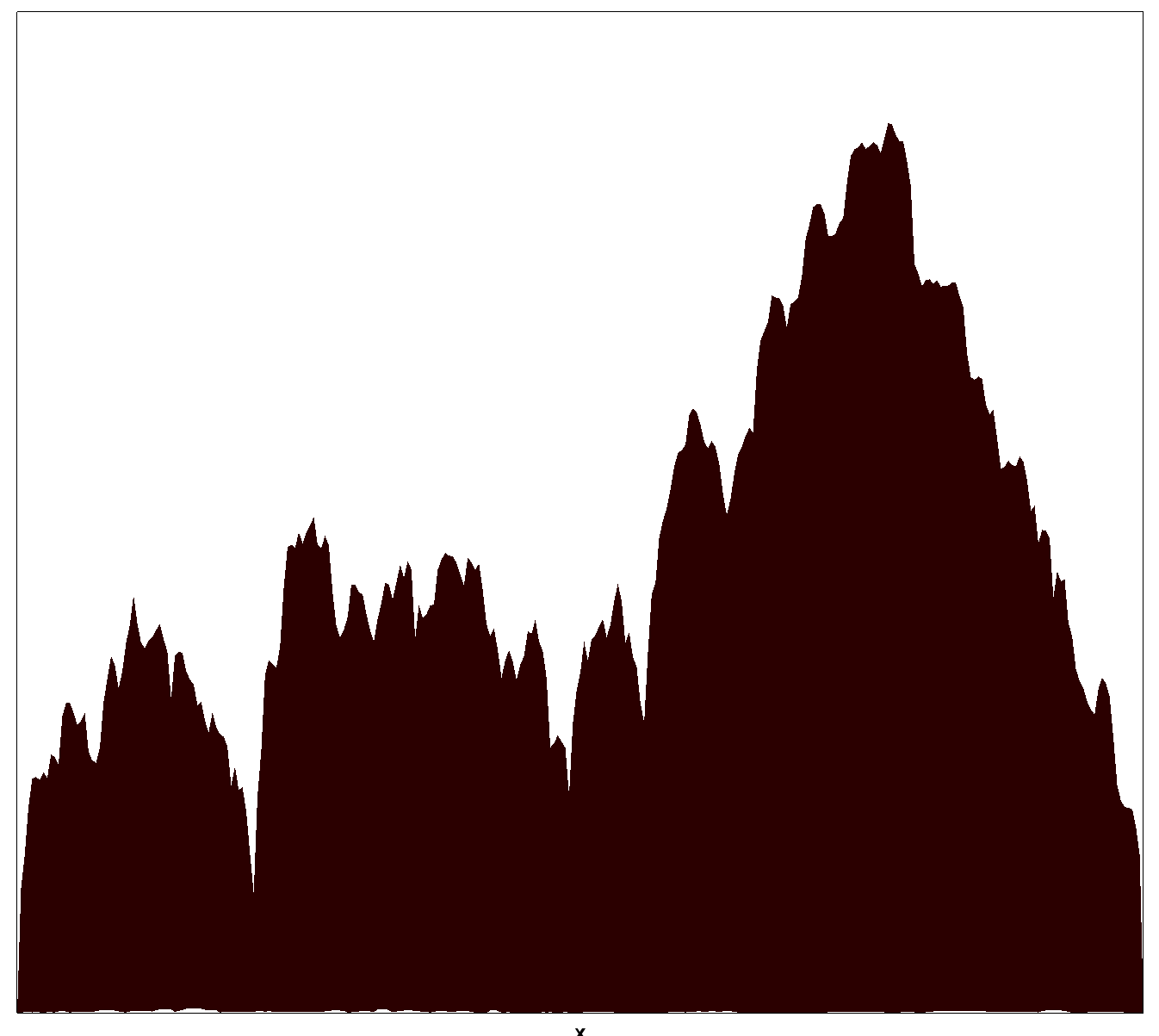} 
    \caption{A simulation of the process $\cal Y$: on the left, the process $\cal Y$; on the right, ${\cal Y}(\cdot,1)$ which is distributed like $\cal Z$.}
    \label{fig:label} 
\end{figure}
  
Recall the definitions of the fields $Y_n^\geq$ and $Y_n^<$ given in Section~\ref{sec:strategy}. Note that $ Y_n^\geq + Y_n^< = Y_n$.

\begin{proposition} \label{prop:convY} Let $\mathcal Y := \mathcal Y^\varnothing$. In probability and with convergence of all moments, 
$$\|n^{-\beta} Y_n - \mathcal Y\| \xrightarrow[n\to\infty]{} 0.$$
In the same sense,
$$\|n^{-\beta} Y^\geq_n - \mathcal Y/2\| \xrightarrow[n\to\infty]{} 0, 
\qquad \text{and}\qquad 
\|n^{-\beta} Y^<_n - \mathcal Y/2\| \xrightarrow[n\to\infty]{} 0.$$
\end{proposition}

\begin{proof}
For $v \in \mathbb T$, let $(a^v_n)_{n \geq 1}$ be the (maximal) increasing sequence of indices with such that $X_{a^v_n} \in Q_v$. Then if $X_n=(x^1_n, x^2_n)$, $n\ge 1$, we set, for each $n\ge 1$,
$X^v_n = (\varphi_v(x^1_{a^v_n}), \varphi'_v(x^2_{a^v_n}))$. This sequence plays the same role in the construction of the subtree rooted at $v$ as the sequence $(X_i)_{i\ge 1}$ in the entire tree (the subtree rooted at $\varnothing$). In other words, if we define $C_n^v(t)$ and $Y_n^v(t)$  analogously to $C_n(t)$ and $Y_n(t)$ but relying on this sequence, we obtain 
\begin{align}\label{conv_mom} 
\Ec{\|n^{-\beta} C_n^v(t) - \mathcal Z^v(t)\|^p} \to 0 
\end{align}
for all $p > 0$. For $r=1, \ldots, 4, $ set $A_{r,n} = N_r /n$ and $\tilde Y_n^r = Y_n^r / n^\beta$. 
By construction, almost surely, for all $t,s \in [0,1]$, 
\begin{align*} 
& \left(\frac{Y_{n}(t,s)}{n^{\beta}} - \mathcal Y(t,s)\right)^2 \\
& = \sum_{r=1}^4 \mathbf 1_{Q_{r}} (t,s) \left(A_{r,n}^\beta \tilde Y_{N_r}^{r}(\varphi_{r}(t),\varphi_{r}'(s)) - A_r^\beta \mathcal Y^{r}(\varphi_{r}(t),\varphi_{r}'(s))\right)^2 \\
& + \mathbf 1_{Q_{2}}(t,s) \left(A_{1,n}^\beta \tilde Y_{N_1}^{1}(\varphi_{1}(t),1) - A_1^\beta \mathcal Y^{1}(\varphi_{1}(t),1)\right)^2 \\ 
& +  \mathbf 1_{Q_{4}}(t,s) \left(A_{3,n}^\beta \tilde Y_{N_3}^{3}(\varphi_{3}(t),1) - A_3^\beta \mathcal Y^{3}(\varphi_{3}(t),1)\right)^2 \\ 
& + 2 \mathbf 1_{Q_{2}}(t,s) \left(A_{1,n}^\beta \tilde Y_{N_1}^{1}(\varphi_{1}(t),1) - A_1^\beta \mathcal Y^{1}(\varphi_{1}(t),1)\right) \left(A_{2,n}^\beta \tilde Y_{N_2}^{2}(\varphi_{2}(t),\varphi_{2}'(s)) - A_2^\beta \mathcal Y^{2}(\varphi_{2}(t),\varphi_{2}'(s))\right)\\
& +  2 \mathbf 1_{Q_{4}}(t,s) \left( A_{3,n}^\beta \tilde Y_{N_3}^{3}(\varphi_{3}(t),1) - A_3^\beta \mathcal Y^{3}(\varphi_{3}(t),1) \right) \left(A_{4,n}^\beta \tilde Y_{N_4}^{4}(\varphi_{4}(t),\varphi_{4}'(s)) - A_4^\beta \mathcal Y^{4}(\varphi_{4}(t),\varphi_{4}'(s))\right).
\end{align*} 
Thus, as in the proof of Proposition \ref{prop:limity}, we obtain
\begin{align*} 
\mathbb E \left[\left(\frac{Y_{n}(t,s)}{n^{\beta}} - \mathcal Y(t,s)\right)^2\right] & \leq 4 \Ec{ \|A_{1,n}^\beta \tilde Y_{N_1}^{1} - A_1^\beta \mathcal Y^{1}\|^2}\\
& + 4 \sqrt{\mathbb E \bigg[ \sup_{t \in [0,1]} |A_{1,n}^\beta \tilde Y_{N_1}^{1}(t,1) - A_1^\beta \mathcal Y^{1}(t,1)|^2\bigg] \cdot \Ec{ \|A_{2,n}^\beta \tilde Y_{N_2}^{2}
 - A_2^\beta \mathcal Y^{2}\|^2}}.
\end{align*}
Next, note that, again using the Cauchy-Schwarz inequality,
\begin{align*} 
& \Ec{ \|A_{1,n}^\beta \tilde Y_{N_1}^{1} - A_1^\beta \mathcal Y^{1}\|^2} \\
& = \Ec{ \|( A_{1,n}^\beta \tilde Y_{N_1}^{1} - A_{1,n}^\beta \mathcal Y^{1}  + A_{1,n}^\beta \mathcal Y^{1} - A_1^\beta \mathcal Y^{1})^2 \|} \\
& \leq \Ec{A_{1,n}^{2\beta} \| \tilde Y_{N_1}^1 - \mathcal Y^1 \|^2} + 2 \sqrt{\Ec{\| \tilde Y_{N_1}^1 - \mathcal Y^1 \|^2} \cdot \E{\|\mathcal Y\|^2} \Ec{(A_{1,n}^\beta-A_1^\beta)^2} } + \E{\|\mathcal Y\|^2} \Ec{(A_{1,n}^\beta-A_1^\beta)^2}.
\end{align*}
Let  $\gamma_n := \Ec{\|n^{-\beta} Y_{n} - \mathcal Y \|^2}$. As $A_{1,n} \to A_1$ almost surely by the concentration of the binomial distribution 
and $\|\tilde Y_n^1 - \mathcal Y^1\| \to 0$ in probability by \eqref{conv_mom} and both convergences are with respect to all moments, 
it follows from the last two displays that
$$\gamma_n \leq 4 \Ec{\gamma_{N_1} A_{1,n}^{2\beta}} + o(1)(1 + \E{\gamma_{N_1}})$$
From here, since $A_{1,n} \to A_1$ and $\gamma = 4\E{A_1^{2\beta}} < 1$ (with $\gamma$ given in \eqref{gamma}), an easy induction on $n$ shows that $(\gamma_n)_{n\geq 1}$ is a bounded sequence. 
The verification that $\gamma_n \to 0$ is now standard in the framework of the contraction method. As in the previous proposition, since the convergence \eqref{conv_mom} also holds for higher moments, by induction over $p$, one verifies the convergence of the $p$th moment. 
Analogously, relying on Proposition \ref{prop:onesidedpm}, one then shows that
$n^{-\beta p } \Ec{ \|Y_n^\geq - Y_n^< \|^p} \to 0$ for all $p > 0$, thereby concluding the proof.
\end{proof}

\subsection{Properties of the limit process $\cal Y$}

\begin{proposition} \label{prop:meanY}
For every $t,s\in [0,1]$, we have $\Ec{\mathcal Y(t,s)} = K_1 h(t) g(s)$ where $g$ is the unique bounded and measurable function on $[0,1]$ satisfying 
\eqref{fix:g}. The function $g$ is continuous and monotonically increasing. 
\end{proposition}
\begin{proof}
Write $\nu(t,s) = \Ec{\mathcal Y(t,s)}$. Taking the expectation in \eqref{fix:lim2} and using the fact that $\nu(t,1) = K_1 h(t)$ yields $\nu = G(\nu)$ where $G$ is the functional operator given by
\begin{equation} \label{fixnu}
\begin{aligned}
G(f)(t,s) 
& = \int_{t}^1 \!\!\!\!\int_{s}^1 (uv)^\beta f\!\left( \frac t u, \frac s v\right) dvdu 
+ \int_0^t \!\!\!\int_{0}^s ((1-u)(1-v))^\beta f \!\left( \frac {t-u}{ 1-u}, \frac {s-v} {1-v}\right) dv du \\
& \quad + \int_t^1 \!\!\!\!\int_{0}^s (u(1-v))^\beta f \!\left( \frac t u, \frac {s-v} {1-v}\right) dv du+\int_0^t \!\!\!\int_{s}^1 ((1-u)v)^\beta f \!\left( \frac {t-u}{1-u}, \frac {s} {v}\right) dv du\\
 & \quad + K_1 \int_{t}^1 \!\!\!\!\int_{0}^s (u(1-v))^\beta h \!\left( \frac t u \right) dv du + K_1 \int_0^t \!\!\!\int_{0}^s ((1-u)v)^\beta h \!\left( \frac {t-u}{ 1-u}\right) dv du.\\
\end{aligned}
\end{equation}
It should be clear from the structure of the terms on the right-hand side that the summands factorize if $f$ is proportional to $h(t) y(s)$ for a bounded, measurable function $y$. More precisely, if $h\otimes y$ denotes the function on $[0,1]^2$ such that $h\otimes y (t,s)=h(t)y(s)$, we have
\begin{align*}
G(K_1 h \otimes y)(t,s) = K_1 \Bigg(\int_{t}^1 u^\beta h & \!\left( \frac t u \right) du +  \int_{0}^t (1-u)^\beta h\!\left( \frac {t-u} {1-u} \right) du \Bigg) \\
& \times  \left( \int_s^1 v^\beta y \!\left(\frac s v \right) dv + \int_0^s (1-v)^\beta y \!\left(\frac {s-v} {1-v} \right) dv + \int_0^s v^\beta dv  \right).
\end{align*}
As $\Ec{\mathcal Z(t)} = K_1 h(t)$, the fixed-point equation \eqref{fix:lim} implies that the first factor in the last display equals $K_1 (\beta+1) h(t) / 2$. (The details of the calculation are worked out in  \cite[Lemma 8]{BrNeSu13}.) Hence, if we choose $y=g$ with $g$ as in the statement of the proposition, then the last display equals $K_1 h(t) g(s)$. To conclude the proof if suffices to show that:
\begin{enumerate}
\item [(i)] there exists at most one fixed-point of $G$ in the set of bounded measurable functions  on $[0,1]^2$, 
\item [(ii)] there exists a unique bounded measurable function $g$ satisfying \eqref{fix:g}, and
\item [(iii)] the function $g$ in (ii) is continuous and increasing.
\end{enumerate}
The first two claims follow from standard contraction arguments. 
We start with (ii). Let  $G^*$ be the operator 
$$G^*(y)(s) =  \frac{\beta + 1}{2} \int_s^1 v^\beta y \left(\frac s v \right) dv + \int_0^s (1-v)^\beta y \left(\frac {s-v} {1-v} \right) dv + \frac{s^{\beta+1}}2.$$
Let $g_1, g_2$ be bounded measurable functions. Observing that the map $s \mapsto \int_s^1 v^\beta  dv + \int_0^s (1-v)^\beta  dv$ considered on $[0,1]$ attains its maximum which has value $(2 - 2^{-\beta}) / (\beta + 1)$ at $s=1/2$, we obtain
\begin{align*}
\|G^*(g_1) - G^*(g_2)\|  
& = \frac{\beta + 1}{2} \cdot \sup_{0\le s\le 1} \left| \int_s^1 v^\beta (g_1-g_2) \left(\frac s v \right) dv + \int_0^s (1-v)^\beta (g_1-g_2) \left(\frac {s-v} {1-v} \right) dv \right | \\
& \leq \frac{\beta + 1}{2}  \|g_1 - g_2 \| \cdot \sup_{0\le s\le 1} \left| \int_s^1 v^\beta dv + \int_0^s (1-v)^\beta  dv \right| \\
& \leq (1 - 2^{-\beta - 1}) \|g_1 - g_2 \|.
\end{align*}
As $1 - 2^{-\beta - 1} < 1$ and the space of bounded measurable functions on $[0,1]$ is complete with respect to the supremum norm, it follows from Banach's fixed-point theorem (see, e.g.\ \cite[Chapter 1]{kirk}) that there exists a unique bounded measurable solution of \eqref{fix:g}. Continuity of this solution follows easily from the theorem of dominated convergence. Monotonicity follows once we have verified (i) since, for any $n \geq 1$, the process $Y_n(t,s)$ (and hence its mean) is increasing in $s$. We move on to (i). 
Let $f_1, f_2$ be bounded and measurable functions on $[0,1]^2$. Then, we have
\begin{align*}
\big|G(f_1)(t,s)- G(f_2)(t,s)\big|^2  & = \left | \E{ \sum_{r=1}^4 \mathbf 1_{Q_r}(t,s) A_r^\beta \cdot (f_1 - f_2)(\varphi_r(t), \varphi_r'(s)) } \right |^2 \\
& \leq \left | \E{ \sum_{r=1}^4 \mathbf 1_{Q_r}(t,s) A_r^\beta \cdot \big|(f_1 - f_2)(\varphi_r(t), \varphi_r'(s))\big|  } \right |^2 \\
 & \leq  \E{ \sum_{r=1}^4 \mathbf 1_{Q_r}(t,s) A_r^{2\beta} \cdot \big|(f_1 - f_2)(\varphi_r(t), \varphi_r'(s))\big|^2 } \\
& \leq \gamma \|f_1-f_2\|^2, 
\end{align*}
where we used Jensen's inequality in the third step, and $\gamma$ is the constant defined in \eqref{gamma}. As $\gamma < 1$, taking the supremum over $t,s$ on the left-hand side shows that $G$ has at most one fixed-point. This proves (i). 
\end{proof}

To conclude the section, we show that the distribution of $\mathcal Y$ is characterized by the identity \eqref{fix:lim2}.
\begin{proposition} \label{prop:idY}
Up to a multiplicative constant, the process $\mathcal Y$ is the unique continuous field (in distribution) with $\Ec{\|\mathcal Y\|^2} < \infty$ satisfying the stochastic fixed-point equation
\begin{align}\label{fix:Y1} 
\mathcal Y \stackrel{d}{=}  \sum_{r=1}^4 B_r(\mathcal Y^{(r)}), 
\end{align}
with operators $B_1, \ldots, B_4$ defined in \eqref{opB}, where $\mathcal Y^{(1)}, \ldots, \mathcal Y^{(4)}, U, V$ are independent and $\mathcal Y^{(1)}, \ldots, \mathcal Y^{(4)}$ are distributed like $\mathcal Y$.
\end{proposition}
\begin{proof} While most parts of the functional contraction method in \cite{NeSu} are developed in the space $\Cone$ (and the space of c{\`a}dl{\`a}g functions on $[0,1]$) many results are formulated for general separable Banach spaces. Let $\mathcal M$ be the set of probability distributions on $\Ctwo$. 
Solutions of \eqref{fix:Y1} in the space $\mathcal M$ are then fixed-points of the map $T : \mathcal M \to \mathcal M$ such that for $\nu\in \cal M$, we have 
\[
T(\nu) = \mathcal L \bigg(\sum_{r=1}^4 B_r(Z^{(r)}) \bigg),
\]
where we write $\mathcal L(\,\cdot\,)$ for the distribution of a random variable, $Z^{(1)}, \ldots, Z^{(4)}$, $U$, $V$ are independent, and $Z^{(1)}, \ldots, Z^{(4)}$ are distributed like $\nu$.

Let $\mathcal M_2$ denote the set of $\nu \in \mathcal M$ with $\int \| f \|^2 d\nu(f) < \infty$.
For a probability distribution $\nu \in \mathcal M_2$, let $\mathcal M_2(\nu)$ be the set of probability distributions $\nu' \in \mathcal M_2$  satisfying 
\begin{align}\label{cond1} 
\Ec{\psi(Z)} = \Ec{\psi(Z')} \quad \text{for all } \psi \in \Ctwo^*, 
\end{align}
where $\mathcal L(Z) = \nu$, $\mathcal L(Z') =\nu'$ and $\Ctwo^*$ denotes the topological dual space of $\Ctwo$, that is, the space of linear maps $\psi : \Ctwo \to \mathbb R$ with
\[\|\psi\| := \sup_{f \in \Ctwo, \|f\|=1}  |\psi(f)| < \infty.\]
By \cite[Lemma 18]{NeSu} (applied with $s=2$ in the notation there), a sufficient condition to ensure that there exists at most one fixed-point of $T$ in $\mathcal M_2(\mathcal L(\mathcal Y))$ is that 
\begin{enumerate}
\item [(i)] $T(\mathcal M_2(\mathcal L(\mathcal Y))) \subseteq \mathcal M_2(\mathcal L(\mathcal Y))$, and
\item [(ii)] $\sum_{r=1}^4 \Ec{\|B_r\|^2} < 1$.
\end{enumerate}
The second claim is easy to verify as $\|B_r\| = A_r^\beta$, and the sum therefore equals $\gamma$ defined in \eqref{gamma}. To prove (i), since $\|B_r(Z^{(r)})\| \leq \|B_r\| \cdot \|Z^{(r)}\|$ and these factors are independent, a simple application of Minkowski's inequality shows $T(\mathcal M_2) \subseteq \mathcal M_2$. Next, we recall the Riesz representation theorem (see, e.g.\ \cite[Theorem IV.6.3] {dun_sch}): for $\psi \in \Ctwo^*$ there exists a (unique) finite signed measure $\eta$ on $[0,1]^2$ such that $\psi(f) = \int f(t) d\eta(t)$. Thus, by Fubini's theorem, condition \eqref{cond1} is satisfied if $Z$ and $Z'$ have the same mean functions. As we already know by Proposition~\ref{prop:limity} that $\mathcal L( \mathcal Y)$ is a fixed-point of $T$ in the set $\mathcal M_2(\mathcal L(\mathcal Y))$, the map $T$ preserves the mean function on this set. This finishes the proof of (i). 

To conclude the proof of the proposition, it remains to show that the mean function 
$m_\eta(t,s) := \int f(t)g(s) d \eta(f,g)$
of a fixed point $\eta$ of $T$ with $\eta \in \mathcal M_2$ is equal to the mean function of $\mathcal Y$ up to a multiplicative constant. By homogeneity, it suffices to consider the case when
$\Ec{m_\eta(\xi,1)} = \Ec{\mathcal Y(\xi, 1)} = \kappa$ with $\kappa$ defined in \eqref{limit_unif_const}.
The claim follows from the previous proof as the map $G$ defined in \eqref{fix:g} has at most one fixed-point in $\Ctwo$ upon verifying that $m_\eta(t,1) = K_1 h(t)$. The map $f(t) = m_\eta(t,1)$ satisfies $\Ec{f(\xi)}=\kappa$ and
$$f(t) = \frac{2}{\beta+1} \left( \int_{t}^1 u^\beta h  \left( \frac t u \right) du +  \int_{0}^t (1-u)^\beta h\left( \frac {t-u} {1-u} \right) du \right), \quad t \in [0,1].$$
By the argument in \cite[Section 5]{CuJo2010}, the unique continuous (or only bounded and measurable) function satisfying these two properties is $K_1 h(t)$. This concludes the proof.
\end{proof}

\section{Proofs of the main results}
\label{sec:proofs}

The representation in Lemma~\ref{lem:decomp} requires us to also consider partial match queries when the first coordinate is arbitrary, and the second should equal a specified value. Of course, all the results we have devised in Sections~\ref{sec:partial_match},~\ref{sec:one_sided} and~\ref{sec:constrained} apply by symmetry, and we only need to set the notations to avoid confusions.

To this end, for $i \geq 1$, let the point $\,\overline{\! X}_i\in [0,1]^2$ be obtained from $X_i$ be swapping the two coordinates. Switching from $(X_i)_{i\ge 1}$ to $(\, \overline{\!X}_i)_{i\ge 1}$ does not alter the shape of the trees: there is a consistent family of relabellings that transforms $(\overline{T}_n)_{n\ge 1}$ into $(T_n)_{n\ge 1}$. However, the corresponding partitions of the unit square are modified (and obtained from one another by a simple symmetry with respect to the principle diagonal of $[0,1]^2$). More specifically, for each $v\in \mathbb T$, the region $\,\overline{\!Q}_v$ is obtained from $Q_v$ by swapping the coordinates of the four corners. Further, for the time-transformations defined in \eqref{phi1} and \eqref{phi2}, we have $\overline \varphi_v = \varphi'_v$ and $\overline \varphi'_v = \varphi_v$. In particular, 
$\,\overline{\!A}_v = A_v$ for all $v \in \mathbb T$. We define the operators $\,\overline{\!B}_r^v, r=1, \ldots, r$ analogously to $B_r^v$ in \eqref{opB} upon replacing $Q_v, \varphi_v$ and $\varphi'_v$ there by their analogues $\overline Q_v, \overline \varphi_v$ and $\overline \varphi'_v$.  

Finally, we shall define the processes $\,\overline{\!C}_n$, $\overline{Y}_n$ and their one-sided versions in the process with swapped coordinates analogously to $C_n, Y_n$.
Of course, all results proved above hold analogously for these quantities, and we denote the corresponding limits by $\overline{\mathcal Z}, \overline{\mathcal Y}$. The joint distributions of these quantities are intricate, but we can characterize them by distributional  fixed-point equations.

\begin{proposition} \label{prop:MMYY}
(a) Up to a multiplicative constant, the pair $(\mathcal Z, \,\overline {\!\mathcal Z})$ is the unique $\Cone^2$-valued process (in distribution) satisfying $\Ec{\|\mathcal Z \|^2}$, $\Ec{\|\,\overline{\!\mathcal Z}\|^2} < \infty$ and
\begin{align} \label{fix:limMM} 
(\mathcal Z, \, \overline{\!\mathcal Z}) 
\stackrel{d}{=} 
\sum_{r=1}^4 A_{r}^\beta \left( \mathbf 1_{Q_{r}^{(1)}}(\,\cdot\,)     \mathcal Z^{(r)}(\varphi_{r}(\,\cdot\,)) , \mathbf 1_{Q_{r}^{(2)}}(\,\cdot\,)    \,\overline{\!\mathcal Z}{}^{(r)}(\varphi'_{r}(\,\cdot\,)) \right). 
\end{align} 
Here, $(\mathcal Z^{(1)}, {\,\overline{\!\mathcal Z}{}^{(1)}}),  \ldots, (\mathcal Z^{(4)}, \, \overline{\!\mathcal Z}{}^{(4)})$ are independent copies of $(\mathcal Z, \,\overline{\!\mathcal Z})$, independent of  $U, V$.

\noindent (b) Similarly, up to a multiplicative constant, the pair $(\mathcal Y, \,\overline{\!\mathcal Y})$ is the unique $\Ctwo^2$-valued process satisfying $\Ec{\|\mathcal Y \|^2}$, $\Ec{\|\,\overline{\!\mathcal Y}\|^2} < \infty$ and
\begin{align} \label{fix:limYY} 
(\mathcal Y, \overline{\mathcal Y}) \stackrel{d}{=} 
\sum_{r=1}^4 \left(B_{r} (\mathcal Y^{(r)}),  \,\overline{\!B}_{r} (\,\overline{\!\mathcal Y}{}^{(r)})\right). 
\end{align} 
Here, $(\mathcal Y^{(1)}, \,\overline{\!\mathcal Y}{}^{(1)}),  \ldots, (\mathcal Y^{(4)}, \,\overline{\!\mathcal Y}{}^{(4)})$ are independent copies of $(\mathcal Y, \,\overline{\!\mathcal Y})$, independent of  $U, V$.
\end{proposition}
\begin{proof} Both statements are proved following the lines of the proof of Proposition \ref{prop:idY}. To show the first, define $\mathcal M, \mathcal M_2$ and $\mathcal M_2(\nu)$ as in that proof, but in the space $\Cone^2$. The only step of that proof one needs to take a closer look at is the verification that, for $\nu, \nu' \in \mathcal M_2$, we have $\nu' \in \mathcal M_2(\nu)$ (that is, condition \eqref{cond1} holds), if $\Ec{Z_i(t,s)} = \Ec{Z'_i(t,s)}$ for $i=1,2$ and $t, s \in [0,1]$, where $\mathcal L((Z_1, Z_2)) = \nu$ and $\mathcal L((Z'_1, Z'_2)) = \nu'$. As in the proof of Proposition~\ref{prop:idY}, this follows from a simple application of Fubini's theorem since every bounded linear form $\psi \in (\Ctwo^2)^*$  can be written as $\psi(f,g) = \int f d \eta_1 + \int g d \eta_2$ for finite signed measures $\eta_1, \eta_2$ on $[0,1]^2$.
The second assertion for the process $(\mathcal Y,  \,\overline{\!\mathcal Y})$ follows analogously.
\end{proof}

\begin{proof} [Proofs of Theorem \ref{thm_main} and Proposition \ref{prop:mean_main}]
Recall the representation in Lemma~\ref{lem:decomp}. We are now ready to make formal the arguments at the end of Section~\ref{sec:strategy}: Theorem \ref{thm_main} follows with 
\begin{align} \label{limitO} 
\cO(a,b,c,d) 
= \frac {\mathcal Y(b,d) - \mathcal Y(b,c) + \mathcal Y(a,d) - \mathcal Y(a,c) + \,\overline{\!\mathcal Y}(d,b) - \,\overline{\!\mathcal Y}(d,a) + \,\overline{\!\mathcal Y}(c,b) -  \,\overline{\!\mathcal Y}(c,a)}2,
\end{align}
since
\begin{enumerate}
\item [(i)] the summands involving $Y_n^\geq, Y_n^<, \overline Y{}^\geq_n$ and $\overline{Y}{}^<_n$ converge uniformly after rescaling by Proposition \ref{prop:convY},
\item [(ii)] $\| n^{-\beta} D^{(i)}_n \| \to 0$ for $i=1, \ldots, 4$, in probability and with convergence of moments since  $D^{(i)}_n$ is bounded from above by the height of $T_n$ whose $p$th moment for $p \geq 1$ is well-known to be $O((\log n)^p)$ \cite[see, e.g,.][]{Devroye1987}, and  
\item [(iii)] 
$\| n^{-\beta} (N_n - n \vol ) \| \to 0$ in probability and with convergence of moments which follows easily from the extension of the Dvoretzky--Kiefer--Wolfowitz inequality to probability distributions on $\mathbb R^2$ in \cite{kie_wol} stating that there exist $c, C > 0$ such that
$$\sup_{n \geq 1} \Prob{ \sup_{0 \leq t, s \leq 1} \left  | \frac{N_n(0,t,0,s) - t s n}{\sqrt{n}} \right | \geq y} \leq C e^{-c y^2}, \quad y \geq 0.$$


\end{enumerate}
Proposition \ref{prop:mean_main} follows from \eqref{limitO} and Proposition \ref{prop:meanY}.
\end{proof}
To characterize the limit process $\cO$, we consider the distributional decomposition of $O_n$. As for the process $Y_n$, this requires the definition of four linear operators, this time on the space $\Cfour^+$.
We set 
\begin{equation} \label{opD}
\begin{aligned}
D_1(f)(a,b,c,d) = A_{1}^\beta \big[ & \mathbf 1_{Q_{1}} (a,c) \mathbf 1_{Q_{1}} (b,d)    f(\varphi_{1}(a), \varphi_{1}(b), \varphi'_{1}(c), \varphi'_{1}(d)) \\
& + \mathbf 1_{Q_{1}} (a,c) \mathbf 1_{Q_{2}} (b,d)    f(\varphi_{1}(a), \varphi_{1}(b), \varphi'_{1}(c), 1) \\
& + \mathbf 1_{Q_{1}} (a,c) \mathbf 1_{Q_{3}} (b,d)    f(\varphi_{1}(a), 1, \varphi'_{1}(c), \varphi'_{1}(d)) \\
& + \mathbf 1_{Q_{1}} (a,c) \mathbf 1_{Q_{4}} (b,d)    f(\varphi_{1}(a), 1, \varphi'_{1}(c), 1) \big] \\ 
D_2(f)(a,b,c,d) = A_{2}^\beta \big[ & \mathbf 1_{Q_{2}} (a,c) \mathbf 1_{Q_{2}} (b,d)    f(\varphi_{2}(a), \varphi_{2}(b), \varphi'_{2}(c), \varphi'_{2}(d)) \\
& + \mathbf 1_{Q_{1}} (a,c) \mathbf 1_{Q_{2}} (b,d)    f(\varphi_{2}(a), \varphi_{2}(b), 0, \varphi'_{2}(d)) \\
& + \mathbf 1_{Q_{2}} (a,c) \mathbf 1_{Q_{3}} (b,d)    f(\varphi_{2}(a), 1, \varphi'_{2}(c), \varphi'_{2}(d)) \\
& + \mathbf 1_{Q_{1}} (a,c) \mathbf 1_{Q_{4}} (b,d)    f(\varphi_{2}(a), 1, 0, \varphi'_{2}(d)) \big] 
\end{aligned}
\end{equation}
\begin{equation}\label{opD2}
\begin{aligned}
D_3(f)(a,b,c,d) = A_{3}^\beta \big[ & \mathbf 1_{Q_{3}} (a,c) \mathbf 1_{Q_{3}} (b,d)    f(\varphi_{3}(a), \varphi_{3}(b), \varphi'_{3}(c), \varphi'_{3}(d)) \\
& + \mathbf 1_{Q_{1}} (a,c) \mathbf 1_{Q_{3}} (b,d)    f(0, \varphi_{3}(b), \varphi'_{3}(c), \varphi'_{3}(d)) \\
& + \mathbf 1_{Q_{3}} (a,c) \mathbf 1_{Q_{4}} (b,d)    f(\varphi_{3}(a), \varphi_{3}(b), \varphi'_{3}(c), 1) \\
& + \mathbf 1_{Q_{1}} (a,c) \mathbf 1_{Q_{4}} (b,d)    f(0,  \varphi_{3}(b), \varphi'_{3}(c), 1) \big] \\
D_3(f)(a,b,c,d) = A_{4}^\beta \big[ & \mathbf 1_{Q_{4}} (a,c) \mathbf 1_{Q_{4}} (b,d)    f(\varphi_{4}(a), \varphi_{4}(b), \varphi'_{4}(c), \varphi'_{4}(d)) \\
& + \mathbf 1_{Q_{2}} (a,c) \mathbf 1_{Q_{4}} (b,d)    f(0, \varphi_{4}(b), \varphi'_{4}(c), \varphi'_{3}(d)) \\
& + \mathbf 1_{Q_{3}} (a,c) \mathbf 1_{Q_{4}} (b,d)    f(\varphi_{4}(a), \varphi_{4}(b), 0, \varphi'_{3}(d)) \\
& + \mathbf 1_{Q_{1}} (a,c) \mathbf 1_{Q_{4}} (b,d)    f(0,  \varphi_{4}(b), 0, \varphi'_{4}(d)) \big].
\end{aligned}
\end{equation}
 
\begin{proof}[Proof of Proposition \ref{thm2}]  By construction, we have, for every $n\ge 1$,
$$O_n \stackrel{d}{=} \left(\sum_{r=1}^4 D_r(O_{N_r}^{(r)})(a,b,c,d) + 1 \right)_{(a,b,c,d) \in I}$$
with conditions on independence and distributions as in \eqref{rec:pm}. By Theorem~\ref{thm_main} it follows that the limit field $\cO$ is a solution to the fixed-point equation \eqref{fix:O}. By the same argument used in the proofs of Propositions~\ref{prop:idY} and~\ref{prop:MMYY}, one shows that $\mathcal L (\cO)$ is the unique solution of \eqref{fix:O} in the set $\mathcal M_2(\mathcal L(\cO))$, where this set is defined as in the proof of Proposition~\ref{prop:idY} but with $\Cfour^+$ instead of $\Ctwo$. It remains to check that, up to a multiplicative constant, the operator $G^{**}$ given by $G^{**}(f)(a,b,c,d) = \sum_{r=1}^4 \Ec{D_r(f)(a,b,c,d)}$, has a unique fixed-point in the space $\Cfour^+$. This follows from several applications of the contraction arguments used in the proof of Proposition~\ref{prop:idY} whose structure we now describe. Let $\varrho_1, \varrho_2 \in \Cfour^+$ be two fixed-points of $G^{**}$ satisfying
$\Ec{\varrho_1(\xi, \xi, 0,1)} = \Ec{\varrho_1(\xi, \xi, 0,1)} = \kappa$. (a) By \cite[Section 5]{CuJo2010}, this implies $\varrho_1(x,x,0,1) = \varrho_2(x,x,0,1) = K_1 h(x)$. (b) Then, setting $(a,c) = (0,0)$, the contraction argument in the proof of Proposition~\ref{prop:idY} shows that $\varrho_1(0, x,0,y) = \varrho_2(0, x,0,y)$; proceeding analogously for the choices $(a,d) = (0,1)$, $(b,c) = (1,0)$ and $(b,d) = (1,1)$ yields $\varrho_1(0,x,y,1)=\varrho_2(0,x,y,1)$, $\varrho_1(x,1,0,y)=\varrho_2(x,1,0,y)$ and $\varrho_1(x,1,y,1)=\varrho_2(x,1,y,1)$. (c) Then, in the next step, one sets $a=0$ and proves along the same lines that $\varrho_1(0, x,y,z) = \varrho_2(0, x,y,z)$; analogously for $b=1, c=0$ and $d=1$ yields $\varrho_1(x, 1, y,z) = \varrho_2(x,1,y,z)$, $\varrho_1(x,y,0,z) = \varrho_2(x,y,0,z)$ and $\varrho_1(x,y,z,1) = \varrho_2(x,y,z,1)$. Finally, with these identities in hand, one can verify that $\|\varrho_1 - \varrho_2 \| = \|G^{**}(\varrho_1) - G^{**}(\varrho_2)\| \leq \sqrt \gamma \|\varrho_1 - \varrho_2 \|$, where $\gamma$ is defined in \eqref{gamma}, just as for the operator $G$ occurring in the proof of Proposition~\ref{prop:meanY}. Since $\gamma<1$, this gives $\varrho_1 = \varrho_2$ and concludes the proof.
\end{proof}

\section{Orthogonal range queries in random $2$-d trees}
\label{sec:kd}

The $2$-d trees have been introduced by \citet{Bentley1975}. As for quadtrees, the data are partitioned recursively, but the splits in $2$-d trees are only binary; since the data is two-dimensional, one alternates between vertical and horizontal splits, depending on the parity of the level in the tree. Given a sequence of points $p_1,p_2,\ldots \in [0,1]^2$, the tree and the regions associated to each node are constructed as follows. Initially, $T_0$ is an empty tree, which consists of a placeholder, to which we assign the entire square $[0,1]^2$. The first point $p_1$ is inserted in this placeholder, and becomes the root, thereby giving rise to two new placeholders. Geometrically, $p_1$ splits \emph{vertically} the unit square in two rectangles, which are associated with the two children of the root. More generally, when $i$ points have already been inserted, the tree $T_i$ has $i$ internal nodes, and induces a partition of the unit square into $i+1$ regions, each one associated to one of the $i+1$ placeholders of $T_i$. The point $p_{i+1}$ is then stored in the placeholder, say $v$, that is assigned to the rectangle of the partition containing $p_{i+1}$. This operation turns $v$ into an internal node, and creates two new placeholders just below. Geometrically, $p_{i+1}$ divides this rectangle into two subregions that are assigned to the two newly created placeholders; that last partition step depends on the parity of the depth of $v$ in the tree:  if it is odd we partition horizontally,  if it is even we partition vertically. See Figure~\ref{fig:2dtrees} for an illustration. (Of course, one could start at the root with a \emph{horizontal} split, and then splits would occur horizontally at even levels and vertically at odd levels.) 

\begin{figure}[tb]
    \centering
    \includegraphics[width=0.95\textwidth]{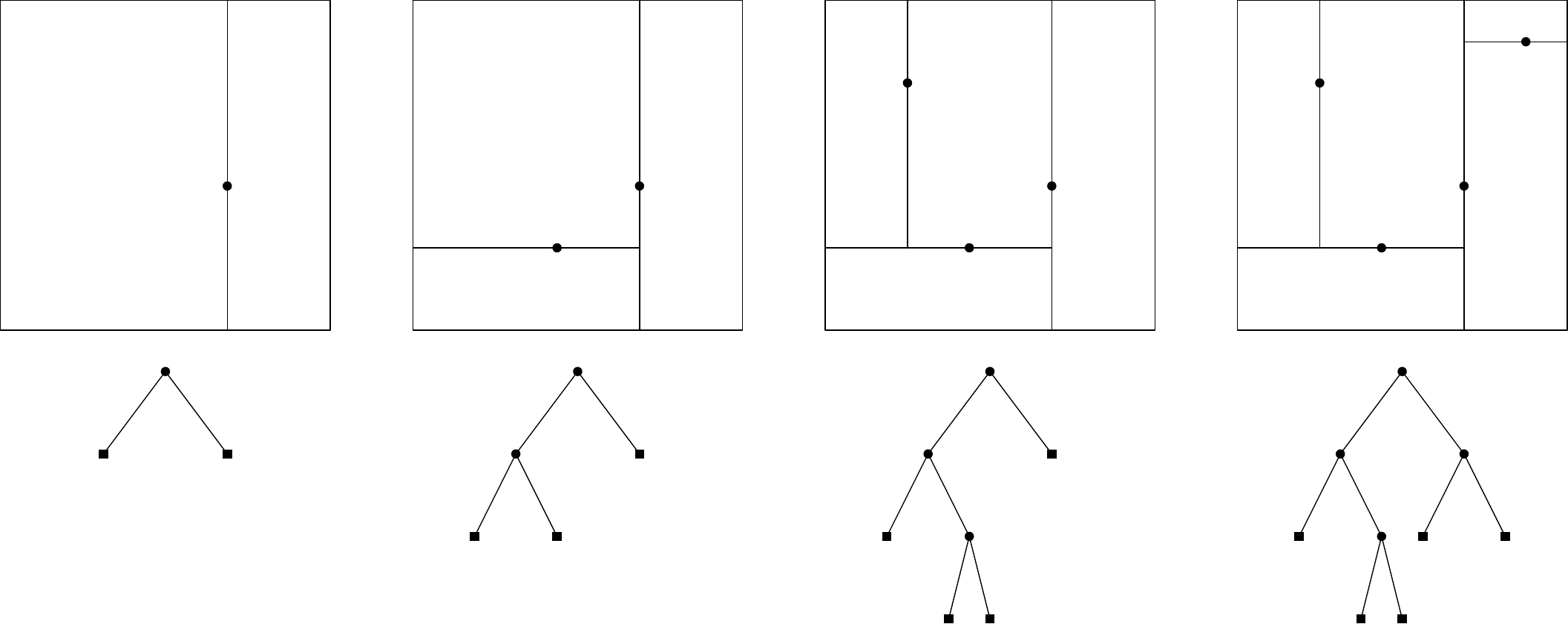}
    \caption{The first four steps of the construction of a $2$-d tree, and the corresponding partitions of the unit square.}
    \label{fig:2dtrees}
\end{figure}

We now consider a sequence $(X_i)_{i\ge1}$ of i.i.d.\ uniform random points in $[0,1]$, and the $2$-d trees obtained by sequential insertion of $X_1,X_2,\dots$ into an initially empty $2$-d tree. It is convenient to consider the trees as subtrees of the infinite binary tree ${\mathbb T}_2 = \cup_{m\ge 0} \{1,2\}^m$. 
We can actually construct at the same time two sequences of trees $(T_n^=)_{n\ge 0}$ and $(T_n^{\perp})_{n\ge 0}$ as well as the corresponding refining partitions of $[0,1]^2$ encoded in the collections $\{Q_v^=: v\in {\mathbb T}_2\}$ and $ \{Q_v^\perp: v\in {\mathbb T}_2\}$, which correspond to the two cases where the split at the root is horizontal or vertical respectively. Put aside the binary splitting, the construction is similar to the one in Section~\ref{sec:partial_match} and we omit the details, and only mention that if $X_1=(U,V)$, then
\[
\arraycolsep=1.4pt
\left\{
\begin{array}{ll}
Q_1^= & = [0,1]\times [0,V] \\
Q_2^= & = [0,1]\times (V,1]
\end{array}
\right.
\qquad \text{and} \qquad 
\left\{
\begin{array}{ll}
Q_1^\perp & = [0,U] \times [0,1] \\
Q_2^\perp & = (U,1] \times [0,1]\,.
\end{array}
\right.
\]

For $(a,b,c,d)\in I$, let $O_n^{=}(a,b,c,d)$ and $O_n^{\perp}(a,b,c,d)$ denote the number of nodes of the $2$-d tree visited to perform the query with rectangle $Q(a,b,c,d)$ when the partition at the root is horizontal or vertical respectively. 
Define \,$\overline{\! O}{}_n^=$ and \,$\overline{\! O}{}_n^\perp$ analogously in the $2$-d tree constructed from the sequence \,$\overline{\! X}_i, i \geq 1$, obtained by swapping the two coordinates of each point. By construction, we have $O_n^{=}(a,b,c,d) = \overline{\! O}{}_n^\perp(c,d,a,b)$ and $O_n^{\perp}(a,b,c,d) = \overline{\! Q}{}_n^=(c,d,a,b)$. In particular, with 
\begin{align} \label{def_swap}
\lambda: I \to I, \quad \lambda(a,b,c,d)  = (c,d,a,b), 
\end{align} 
the sequences $(O_i^=)_{i \geq 1}$ and $(O_i^\perp \circ \lambda)_{i \geq 1}$ are identically distributed. Hence, it suffices to focus on the sequence
$O_i^=, i \geq 1$.

\begin{thm}\label{thm:main_kd}
There exist random continuous $\Cfour^+$-valued random variables ${\cO^{=}}$ and ${\cO^{\perp}}$ (random fields) such that, in probability and with convergence of all moments,
\[\Big \| \frac{O^=_n - n \vol }{n^\beta} - {\cO^=}\Big \| \xrightarrow[n\to\infty]{} 0\,
\qquad\text{and}\qquad
\Big \| \frac{O^\perp_n - n \vol }{n^\beta} - {\cO^\perp}\Big \| \xrightarrow[n\to\infty]{} 0
.\]
Here, the constant $\beta$ is the same as for quadtrees and defined in \eqref{limit_unif_const}. The processes ${\cO^\perp}$ and ${\cO^=} \circ \lambda$ have the same distribution.
\end{thm}

To characterize the limit field $\cO^=$ as the solution to a stochastic fixed-point equation let
us  keep track of the relative positions with respect to cells of the partition as follows: for $s\in [0,1]$, we set 
\[
\psi^=_1(s)= {\mathbf 1}_{[0,V]}(s)\frac{s}{V}
\qquad\text{and}\qquad
\psi^=_2(s)= {\mathbf 1}_{[V,1]}(s)\frac{s-V}{1-V}\,.
\]
Note that the volumes of the rectangular regions at the first level of the partition are 
\[A_{=,1} =  V, \qquad A_{=,2} = 1-V\,.\]
We now define the operators $D_1^=$ and $D_2^=$ 
\begin{equation} \label{eq:epDe}
\begin{aligned}
D_1^=(f)(a,b,c,d) & = A_{=,1}^\beta \left[\mathbf{1}_{[0,V]}(d) f(a,b, \psi_1^=(c), \psi_1^=(d)) + \mathbf{1}_{[0,V]\times (V,1]}(c,d) f(a,b, \psi_1^=(c), 1)\right]\\
D_2^=(f)(a,b,c,d) & = A_{=,2}^\beta \left[{\mathbf 1}_{(V,1]}(c) f(a,b, \psi_2^=(c), \psi_2^=(d))+ \mathbf{1}_{[0,V]\times (V,1]}(c,d) f(a,b, 0, \psi_2^=(d)) \right].
\end{aligned}
\end{equation}

\begin{proposition}\label{pro:fixed_point_kd}
Up to a multiplicative constant, $\cO^=$ is the unique $\Cfour^+$-valued random field (in distribution) such that $\Ec{\|\cO^=\|^2} < \infty$ satisfying the stochastic fixed-point equation
\begin{align} \label{fix:Okd} 
{\cO^=} \stackrel{d}{=} D^=_1(\cO^{=,(1)} \circ \lambda) + D^=_2(\cO^{=,(2)} \circ \lambda)
\end{align}
where $\cO^{=,(1)}$ and $\cO^{=,(2)}$ are copies of $\cO^=$, $D_1^=, D_2^=$ are random linear operators defined in \eqref{eq:epDe} and $\lambda : I \to I$ is defined in \eqref{def_swap}. Furthermore, the random variables 
 $\cO^{=,(1)}, \cO^{=,(2)}$, and $(D_1^=, D_2^=)$ are independent.
\end{proposition}


\medskip

Proposition~\ref{pro:fixed_point_kd} only characterizes the distribution of $\cO^=$ up to a multiplicative constant. The next proposition identifies the limit mean, and hence the missing multiplicative constant. Here, the values of the constants appearing are reminiscent of the fact 
that, for uniform partial match queries in the trees $T_n^=$ and $T_n^\perp$, \citet*{FlPu1986} proved the analogue of expansion \eqref{limit_unif} and 
\citet*{ChHw2006} identified the leading constants as 
\begin{equation*}
\kappa^= = \frac{13 (3-5\beta)}{2} \kappa \qquad \text{and} \qquad \kappa^\perp = 13(2\beta -1) \kappa\,.
\end{equation*}
Note that $2 \kappa^= = (\beta + 1) \kappa^\perp$.

\begin{proposition}\label{prop:mean_kd} Let $(a,b,c,d)\in I$. Then, we have 
\begin{align*}
\Ec{\cO^=(a,b,c,d)} & = \frac{13 (3-5\beta)}{4} \Big( \mu(a,d) - \mu(a,c) + \mu(b,d) - \mu(b,c) \Big)\\
& + \frac{13(2 \beta-1)}{2} \Big(\mu(c,b) - \mu(c,a) + \mu(d,b) - \mu(d,a)\Big),
\end{align*}
where $\mu(t,s)$ is the function from Proposition~\ref{prop:mean_main}.
\end{proposition}




{\small
\setlength{\bibsep}{.3em}
\bibliographystyle{plainnat}
\bibliography{bib_quadtrees} 
}

\end{document}